\documentclass{article}

\usepackage{pgfplots}
\pgfplotsset{compat=1.13}

\usepackage{url}
\usepackage[a4paper]{geometry}
\usepackage{verbatim,amsmath,algorithm,amssymb,amsthm}
\usepackage{subcaption}
\newtheorem{assumption}{Assumption}
\newtheorem{remark}{Remark}
\newtheorem{definition}{Definition}
\newtheorem{lemma}{Lemma}
\newtheorem{example}{Example}
\newtheorem{theorem}{Theorem}
\usepackage{algpseudocode}
\usepackage{multirow}
\usepackage{tikz}
\usepackage{lipsum}
\usepackage{todonotes}
\newtheorem*{lemma*}{Lemma}
\newtheorem{corollary}{Corollary}
\usetikzlibrary{positioning}
\usetikzlibrary{arrows,automata}

\newcommand{\fl}{f^{\ell}}

\newcommand{\fsl}{f^{S_k^\ell}}

\newcommand{\vlmone}{v_k^{\ell-1}}

\newcommand{\xkl}{x_k^\ell}
\newcommand{\skl}{s_k^\ell}

\newcommand{\E}{\mathbb{E}}

\newcommand{\gk}{g_k}

\newcommand{\gradf}{\nabla_x f(x_k)}
\newcommand{\name}{MU$^\ell$STREG}
\newcommand{\nameuno}{MU$^1$STREG}
\newcommand{\namedue}{MU$^2$STREG}
\newcommand{\nametre}{MU$^3$STREG}

\newcommand{\norm}[1]{\left\Vert #1 \right\Vert}
\newcommand{\graffe}[1]{\left\lbrace #1 \right\rbrace}
\newcommand{\quadre}[1]{\left[ #1 \right]}
\newcommand{\tonde}[1]{\left( #1 \right)}
\newcommand{\Sub}{\mathcal{S}} 
\newcommand{\R}{\mathbb{R}}

\newcommand{\mush}{\texttt{MUSH} }
\newcommand{\mnist}{\texttt{MNIST} }
\newcommand{\aninea}{\texttt{A9A} }
\newcommand{\ijcnnone}{\texttt{IJCNN1} }
\newcommand{\fine}{\phi}
\newcommand{\coarse}{\phi}

\newcommand{\regc}{\lambda_k\|\gk\|}

\begin{document}
\title{A  multilevel stochastic regularized first-order method \\ with application to finite sum minimization}


\author{Filippo Marini\thanks{Dipartimento di Ingegneria Industriale, Universit\`a degli Studi di Firenze, Viale Morgagni 40/44, 50134 Firenze, Italia.
Email: {\tt filippo.marini@unifi.it}}\and{\,}
 Margherita Porcelli$^{\dagger,}$\thanks{Dipartimento di Ingegneria Industriale, Universit\`a degli Studi di Firenze, Viale Morgagni 40/44, 50134 Firenze, Italia. Member of the INdAM Research Group GNCS. Email: {\tt margherita.porcelli@unifi.it}}\thanks{ISTI--CNR, Via Moruzzi 1, Pisa, Italia}
 \and Elisa Riccietti\thanks{ENS de Lyon, Univ Lyon, UCBL, CNRS, Inria, LIP, F-69342, LYON Cedex 07, France. Email: {\tt elisa.riccietti@ens-lyon.fr}} }

\maketitle

\begin{abstract}
In this paper, we propose a multilevel stochastic framework for the solution of nonconvex unconstrained optimization problems. The proposed approach uses random regularized first-order models that
exploit an available hierarchical description of the problem, being either in the classical
variable space or in the function space, meaning that different levels of accuracy for the objective function are available. We propose a convergence analysis  showing an almost sure global convergence of the method to a first order stationary point.
The numerical behavior is tested on the solution of finite sum minimization problems.
Differently from classical deterministic multilevel schemes, our stochastic method does not require the finest approximation to coincide with the original objective function along all the optimization process.  This allows for significantly decreasing their cost, for instance in data-fitting problems, where considering all the data at each iteration can be avoided.

\end{abstract}

\section{Introduction}\label{sec_method}

Many modern applications require the solution of large scale stochastic optimization problems, i.e., 
problems of the form:
\begin{equation}\label{pb_opti}
\min_{x \in  \mathbb{R}^n} f(x),
\end{equation}
where $f$ is a function that is assumed to be smooth and bounded from below, whose value can only be computed with some noise \cite{alarie2021two}.
When considering this problem, it is usually assumed that  random realizations of $f$ are computable, with variable accuracy \cite{bergou22,storm}. While this framework is borrowed from
derivative free optimization, it applies to derivative-based optimization as well, cf. \cite{storm}.
 Expected risk  minimization  and problems in which the objective function is the outcome of a stochastic simulation are examples that fit this framework \cite{bottou_review,storm}. 

Methods to address \eqref{pb_opti} usually rely on the assumption that the accuracy of the function approximations grows asymptotically, which leads to iterations that are more and more expensive \cite{bergou22,storm}. Moreover, such methods control the accuracy of the function estimates but not the size of the variables.   
An important challenge in this context is thus to develop scalable stochastic methods, able to handle the increasing costs of large scale stochastic optimization problems. 

In classical scientific computing, \textit{multilevel methods} represent powerful techniques that have been developed to cope with structured large scale optimization problems where the limiting factor is the number of variables. When the structure of the problem at hand allows for a hierarchical description of the problem itself,  these methods reduce the cost of the problem's solution by computing cheap steps by exploiting function approximations defined on subspaces of progressively smaller dimension. Thanks to this, they achieve not only considerable speed-ups but also an improved quality solution in various  applications, spanning from the solution of partial differential equations to image reconstruction \cite{SGratton_ASartenaer_PhLToint_2008,eusipco,lauga2024iml,nash2000}.

Existing multilevel methods are limited to a \textit{deterministic} context and thus are unsuitable to address stochastic optimization problems. Moreover, they have always been used to exploit  hierarchies in the variables space, such as discretizations of  infinite dimensional problems on selected grids. However, in many modern applications the limiting factor can be the accuracy of the function estimates rather than the size of the model. 

 In this work we propose an extension of multilevel methods to a stochastic setting. 
More precisely, we assume to have access to a hierarchy of  randomly chosen computable representations of $f$, built either by reducing the dimension in the variables space or by reducing the noise of the function approximation, or both. 
Our multilevel adaptation decreases the cost of the iterative solution by coupling standard ``fine" stochastic steps with ``coarse"  steps, which are less expensive and keep a low cost for all the iterative process. Such steps can be built either:
\begin{itemize}
\item By reducing the dimension of the problem in the variables space, as in classical multilevel schemes, but possibly in a stochastic way.
\item Using low accuracy approximations of the objective function even in later steps of the optimization process.  
\end{itemize}
A level $\ell$ will correspond to a subset of variables and to a noise level in the function approximation. The fine steps will still be computed considering all the variables and increasingly accurate function approximations, while the coarse steps will be computed by taking into account just small subsets of variables and/or inaccurate function approximations. 
While the fine iterations will still become more and more expensive, as in standard stochastic methods, their number will be reduced, and progress and speed will be ensured by the coarse iterations. 

Our multilevel framework thus extends classical multilevel schemes by allowing  for the solution of stochastic problems such as \eqref{pb_opti}, with hierarchies built also in the function space. Differently from the classical setting, the steps are stochastic, as well as the generated sequence of iterations. 
 The stochasticity of the framework allows one to mitigate one of the main limitations of deterministic multilevel methods: the need of periodically handling the full original problem at fine scale, which limits the size of affordable problems. In this context indeed, as it is the case for classical stochastic methods \cite{bergou22,blanchet2019convergence,byrd2012sample,storm}, this will be necessary only towards the end of the optimization procedure.

Inspired by the stochastic trust-region framework STORM in \cite{storm},  we propose a stochastic multilevel Adaptive Regularization (AR) technique\footnote{We choose to employ AR techniques rather than trust-region techniques as they are easier to adapt to a multilevel context, cf. \cite{calandra2021high}.} named \name{} for MUltilevel STochastic REegularized Gradient method, where $\ell$ refers to the number of levels in the hierarchical description of the problem. As in standard 
 AR techniques, the automatic step selection choice is made possible at each level by the use of models regularized by an adaptive regularization parameter, but the fine level models are stochastic. As for STORM, we prove that our method converges to a first order stationary point as long as 
 the local models of the objective function at the finest level are fully linear and the function estimates are sufficiently accurate, both with sufficiently high probability. 
The proposed analysis extends that proposed in \cite{storm} to adaptive regularization methods while including also the multilevel steps.

 From a practical point of view, we test the multilevel paradigm on the solution of \textit{finite sum  minimization problems}. 
Such problems have their origin in large scale data analysis applications where models depending on a large number of parameters $n$ are fitted to a large set of $N$ samples, thus  either $n$ or $N$ (or both) are really large, see e.g. \cite{bollapragada2018adaptive, bottou_review}.
Several methods have been developed to cope with the large sizes of the datasets. In particular, optimization techniques based on subsampling techniques have been proposed, among them the numerous variations of classical stochastic gradient descent (SGD), see e.g. \cite{bellavia1,bollapragada2018adaptive, bottou_review,saga,johnson2013accelerating} and references therein.

When considering expected risk minimization problems, there is a natural way of building a hierarchy of computable function approximations, through the definition of nested subsample sets and of finite sum minimization problems. 
A multilevel method in this context alternates ``fine steps", i.e., steps computed considering increasingly large  subsets of data and ``coarse steps" computed taking into account just small subsets of data (mini-batches). In this context they can be viewed as a way for accelerating adaptive sampling strategies. Moreover, 
the coarse steps are computed by minimizing a model that is built from the coarse level approximations (finite sum problems defined on mini-batches) by adding a correction term, usually known as ``first order coherence" in the multilevel literature, which (in this context) accounts for the discrepancy between the fine gradient and the coarse one. When the full gradient is used at fine level, this is reminiscent of the term added in the reduced variance gradient estimate of the mini-batch version of SVRG \cite{johnson2013accelerating}. 
 Multilevel methods in this case can thus also be interpreted as variance reduction methods, cf. \cite{alena}, with the advantage of allowing for an automatic choice of the step size.

\paragraph{Contributions}
\begin{itemize}
   \item This is the first stochastic framework for multilevel methods, that are currently limited to the deterministic case.
   \item The multilevel framework allows for hierarchies in the function space, i.e., built by considering function approximations with variable accuracy. 
   
  \item The stochastic multilevel framework mitigates  the limiting factor of classical deterministic multilevel methods, whose convergence theory requires the fine level function to coincide with the original target function at every iteration.
  
  \item Our multilevel framework, and thus the stochastic analysis, also covers the classical one-level case.
  
   \item The multilevel method offers a strategy to accelerate methods for stochastic optimization such as STORM. 
\end{itemize}

\paragraph{Related work}

\textit{Multilevel methods.} As a natural extension of multigrid methods \cite{briggs2000multigrid} to a nonlinear context, multilevel methods were first proposed by Nash through the MG/OPT framework \cite{nash2000} and later extended to trust region schemes \cite{SGratton_ASartenaer_PhLToint_2008}. Recently these methods have been extended to other contexts: high-order models \cite{calandra2021high}, non-smooth optimization \cite{lauga2024iml,parpas}, machine learning \cite{mOFFO,bcd,Kopanicakova_2021g}.  A multilevel method that exploits hierarchies in the function space has been explored in \cite{alena}, for deterministic finite sum convex problems leveraging the multilevel scheme of MG/OPT developed in \cite{nash2000}. 
Recent research \cite{mOFFO} proposes a (deterministic) multilevel version of the OFFO method that does not require function evaluations and that is based on the classical multilevel scheme constructed on the variable space.

\textit{Derivative free optimization} An idea close to that of multilevel methods to alternate between accurate steps and cheap steps using more or less information can be found in full-low evaluation derivative free optimization for direct search methods
 \cite{full_low,full_low2}. Another technique that has been considered to reduce the cost of DFO problems is random subset selection \cite{coralia_subset}. 
 
 \textit{Stochastic regularization methods}
 In \cite{bergou22} the authors propose a Levenberg-Marquardt adaptation of the STORM framework for
 noisy least squares problems. As in our work, the step size in this context is updated through a regularization parameter. We inherited from this work the dependence of the regularization parameter from the norm of the gradient (cf. \eqref{clem} below) and the definition of accurate models (cf. Definition \ref{def:fullylin}). 
The recent literature on variants of the standard trust-region method based on the use of random models is very extensive, we refer to \cite{bandeira2014,bellavia2,bellavia1,sara2024,rinaldi2024} to name a few and references therein. In addition,  stochastic adaptive second order regularization methods are considered in \cite{jin2025sample,scheinberg2023stochastic}, while higher order models are
considered in \cite{bellavia2022adaptive}.

\paragraph{Organization of the paper} In section 
\ref{sec:2level} we introduce our \name{} method and we propose its convergence analysis in section \ref{sec_conv}. In section \ref{sec:multilevel} we specialize the \name{} framework to 
finite sum minimization and we analyze its numerical performance in section \ref{sec:num}.
We draw some conclusions and present some perspectives in section \ref{sec_concl}.


\section{The multilevel stochastic regularized gradient method}\label{sec:2level}

In this section we describe our new MUltilevel STochastic REegularized Gradient method (MU$^{\ell}$\-STREG)\footnote{The $\ell$ denotes the number of levels in the hierarchical problem description.} for the solution of problem (\ref{pb_opti}).

\paragraph{Hierarchical representation of problem \eqref{pb_opti}}
We assume to have at disposal, at each iteration $k$,  a fine level computable  approximation $f_k$ of $f$ and  a hierarchy of coarse computable 
approximations, $\{\fine_k^{\ell}\}_{\ell=1}^{\ell_{\max}-1}$, where $1$ corresponds to the coarsest level. 
More precisely, we assume that, for each $k$,  $\fine_k^{\ell}$ is less costly to compute than $\fine_k^{\ell+1}$ and that $\fine_k^{\ell_{\max}-1}$ is less costly to compute than $f_k$, either because  it
is less accurate  and/or because it is defined on a smaller dimensional space. 
We also assume that the functions $\phi_k^\ell$ are randomly chosen at iteration $k$, but once they have been defined, they are deterministic functions. 
As in classical multilevel methods, we assume to have at disposal some transfer operators $R^{\ell}_k$ (restriction) and $P^{\ell}_k$ (prolongation) to transfer the information (variables and gradients) from level $\ell$ to level $\ell-1$ and vice-versa, such that $R^{\ell}_k=\nu_k (P^{\ell}_k)^T$ for some $\nu_k>0$ \cite{briggs2000multigrid}. 
Unlike the classical framework, such operators can  be random and often vary from one iteration $k$ to another.  If the hierarchy is built just in the functions space all the variables will have the same dimension and the transfer operators will thus just be the identity. We present here some problems that fit this framework.  In the following we will refer to the fine level with the index $\ell_{\max}$.

\begin{example}\emph{Expected risk minimization problems} \cite{bottou_review}
 Assume to have two sets $\mathcal{Z},\mathcal{Y}$, a loss function $\mathcal{L}:\mathcal{Z}\times\mathcal{Y}\rightarrow\mathbb{R}$, and a probability distribution $P$ on $\mathcal{Z}\times\mathcal{Y}$. The expected risk minimization problem is
\begin{equation}\label{pb_expected_risk}
     \min_{x \in \mathbb{R}^n} \mathbb{E}_P[\mathcal{L}(y,m_x(z))]
\end{equation}
 for a given parametric model $m_x:\mathcal{Z}\rightarrow \mathcal{Y}$. Given a data set $\{(z_i,y_i)\}$, $i = 1,\dots, N$, drawn from the distribution $P(z, y)$, the fine approximation at iteration $k$ can be defined as the averaged sum of the the functions $f^{(i)}(x) := \mathcal{L}(y_i,m_x(z_i))$, that is as 
$$
\frac{1}{|\mathcal{S}_k^{\ell_{\max}}|}\sum_{i\in \mathcal{S}_k^{\ell_{\max}}}f^{(i)}(x) \quad x\in\mathbb{R}^n,
$$
over increasingly larger subsets
$\mathcal{S}_k^{\ell_{\max}}\subseteq\{1,\dots,N\}$, such that $\mathcal{S}_{k}^{\ell_{\max}}\subseteq\mathcal{S}_{k+1}^{\ell_{\max}}$ (cf. section \ref{sec:multilevel}).
The hierarchy of functions approximations at coarse levels can then be built in two different ways:

\begin{enumerate}
    \item \emph{ Hierarchy in the samples space} The coarse approximations can be defined, at iteration $k$,  as the averaged sum of the $f^{(i)}$ over 
nested subsets of $\mathcal{S}_k^{\ell_{\max}}$, that is $\fine_k^\ell:= \fsl$ where:
$$
\fsl(x)=\frac{1}{|\mathcal{S}_k^\ell|}\sum_{i\in \mathcal{S}_k^\ell}f^{(i)}(x) \quad x\in\mathbb{R}^n,
$$
for $\mathcal{S}_k^\ell\subseteq\mathcal{S}_k^{\ell+1}\subseteq \mathcal{S}_k^{\ell_{\max}}$ for all $\ell$ and for all $k$.

\item \emph{Hierarchy on the variables}
The coarse approximations can still be defined as the averaged sum of the $f^{(i)}$ over $\mathcal{S}_k^{\ell_{\max}}$, but considering a (possibly random) subset of the variables: 
$$
\phi_k^\ell(\bar x)=\frac{1}{|\mathcal{S}_k^{\ell_{\max}}|}\sum_{i\in \mathcal{S}_k^{\ell_{\max}}}f^{(i)}(\bar x) \quad \bar x\in\mathbb{R}^{n_k^\ell}
$$
with $n^\ell_k\leq n_k^{\ell+1}\leq n$ for all $\ell$ and for all $k$.

\end{enumerate}
Notice that the sets $\mathcal{S}_k^\ell$ are drawn randomly, but once they are drawn, the function approximations are deterministic. 

\end{example}

\begin{example}\emph{Montecarlo simulations} \cite{giles2008multilevel}
Let $\xi$ be a random variable defined on the probability space $(\Omega,P)$. 
Let $f(x)=\mathbb{E}(F(x,\xi))$, with $F:\mathbb{R}^n\times \Omega\rightarrow \mathbb{R}$. $F(x,\xi)$
can represent for instance the fit of the solution of a stochastic PDE to some data. $F$ can be approximated at fine level by the stochastic  approximations 
$F_{h_k}(x,\xi)$, the output of a PDE solver with mesh size $h_k>0$ and  
$f$ can then be approximated by the Montecarlo estimator 
$$
\hat{f}_{N_k,h_k}(x)=\frac{1}{N_k} \sum_{i=1}^{N_k} F_{h_k}(x,\xi_i)
$$
   where $\xi_1,\dots,\xi_{N_k}$ are iid samples drawn from $P$. 
   The coarse approximations can be built either: 
   \begin{enumerate}
       \item \emph{In the function space}, by defining 
       $$
\phi^{\ell}_k(x)=\frac{1}{N_k^\ell} \sum_{i=1}^{N_k^\ell} F_{h^\ell_k}(x,\xi_i)
$$
    with $h_k^\ell\geq h_k^{\ell+1}\geq h_k$ and/or $N_k^{\ell}\leq N_k^{\ell+1}\leq N_k$ for al $\ell$ and for all $k$.    
       \item \emph{In the variables space} by defining 
       $$
\phi^{\ell}_k(\bar{x})=\frac{1}{N_k} \sum_{i=1}^{N_k} F_{h_k}(\bar{x},\xi_i), \quad \bar{x}\in\mathbb{R}^{n_k^\ell}
$$
with $n^\ell_k\leq n_k^{\ell+1}\leq n$ for all $\ell$ and for all $k$.
   \end{enumerate}
\end{example} 

\paragraph{The step computation}
For any level $\ell$ and at each iteration $k$, our multilevel gradient method can choose between two different types of stochastic steps: a gradient step, which is known as the \textit{fine step}, or a \textit{coarse step} computed by exploiting the approximations of $f$.
Notice that,  differently from classical deterministic multilevel schemes, the steps are all stochastic. 

At the finest level, if the fine step is chosen, a model $m_k(x_k + s) := f_k+g_k^Ts$ approximating $f$ around the current iterate $x_k$ is built, where $f_k$ and $g_k$ are approximations of $f(x_k)$ and $\nabla_x f(x_k)$, respectively,   and a classical stochastic gradient step $s_k=-\frac{g_k}{\lambda_k \|g_k\|}$ is taken\footnote{This amounts to minimize $m_k(x_k+s)+\frac{\lambda_k\|g_k\|}{2}\|s\|^2$ wrt $s$. Notice that the stepsize depends on the norm of the gradient as in \cite{bergou22}, cf. discussion in \cite[section 3.1]{bergou22}. } for some $\lambda_k >0$, with $g_k$ a realization of $\nabla_x f(x_k)$.   
   
If the coarse step is chosen at iteration $k$,  the step is found by recursively minimizing a sequence of regularized models  $m_{k}^{R,\ell}$, for $\ell=\ell_{\max}-1,\dots,1$, built exploiting the random approximations $\{\fine_k^{\ell}\}_{\ell=1}^{\ell_{\max}-1}$ of $f$ and are thus either defined in a lower dimensional space, or
employs inaccurate function approximations, or both.
More precisely, at each level $\ell>1$ (including the finest level $\ell_{\max}$) a coarse model is defined for the immediately coarser level, which will serve as an objective function for level $\ell-1$ when the algorithm is called recursively. 
Starting at the fine level and  
considering the highest coarse approximation in the hierarchy $\fine_k^{\ell_{\max}-1}$,  at iteration $k$ we define 
\begin{align*}
\varphi_k^{\ell_{\max}-1}(s)&=\fine_k^{\ell_{\max}-1}(R_k^{\ell_{\max}-1} x_k+s)\\
&+(R_k^{\ell_{\max}-1} g_k-\nabla_s \fine_k^{\ell_{\max}-1}(R_k^{\ell_{\max}-1}x_k))^Ts,
\end{align*}
i.e., $\varphi_k^{\ell_{\max}-1}$ is a modification of the coarse function $\fine_k^{\ell_{\max}-1}$ through the addition of a correction term. This correction aims to enforce the following relation:
\begin{equation*}
\nabla_s \varphi_k^{\ell_{\max}-1}(0)= R_k^{\ell_{\max}-1} g_k,
\end{equation*}
which ensures that the behaviour of the coarse model is coherent with the fine objective function approximation, up to order one. The step is defined as an approximate minimizer of the regularized model 
\begin{equation}\label{clem}
	m^{R,\ell_{\max}-1}_k(s)=\varphi_{k}^{\ell_{\max}-1}(s)+ \frac{\lambda_k^{\ell_{\max}-1}\|g_k\|}{2}\|s\|^{2},
	\end{equation}
   with $\lambda_k^{\ell_{\max}-1}>0$,
by calling the multilevel procedure in a recursive way. At the first level of the recursive call, the algorithm will take as an objective function $m^{R,\ell_{\max}-1}_k$, and take either a gradient step for $m^{R,\ell_{\max}-1}_k$ or a coarse step. To define the coarse step,  a coarse model for $m^{R,\ell_{\max}-1}_k(s)$ is built, which involves the approximation $\fine_k^{\ell_{\max}-2}$, the restriction of the correction vector and the regularization term from level $\ell_{\max}-1$ and a new correction term and a new regularization for level $\ell_{\max}-2$. The procedure is repeated until the coarsest level is reached, which is minimized directly without further recursion. Once a coarse step $s^*$ is found at the end of the recursive procedure, we set $s_k=P_k^{\ell_{\max}}s^*$.

At a generic level $\ell$, the recursive call is stopped as soon as  a step $s_k^{\ell-1}$ is found that satisfies the following conditions:
\begin{align}\label{stop_ell}
m_k^{R,\ell-1}(s_k^{\ell-1}) <m_k^{R,\ell-1}(0),\quad 
	\left\|\nabla_s m_k^{R,\ell-1}(s_k^{\ell-1})\right\| \leq \epsilon^{\ell-1} \|s_k^{\ell-1}\|,
 \end{align}
for some $\epsilon^{\ell-1}>0$, and we set $s_k^\ell:=P_k^\ell s_k^{\ell-1}$.
As we will see, these conditions will ensure the convergence of the multilevel method in the spirit of the Adaptive-Regularization algorithm with a first-order model described e.g., in \cite[Sec. 2.4.1]{book_compl}.
Note that even if we use a first order model at fine level, we could use a higher order method to minimize the lower level models. 

In order to be meaningful, the coarse steps are restricted to iterations such that 
\begin{equation*}
  \| R_k^\ell g^\ell_k\|\geq \kappa^\ell \|g^{\ell}_k\|
\end{equation*}
with $g^\ell_k$ a realization of the gradient of the objective function at level $\ell$, and  for $\kappa^\ell\in(0,\min_k\min\{1,\|R_k^\ell\|\})$ \cite{SGratton_ASartenaer_PhLToint_2008}.

This framework is flexible and encompasses several actual implementations: at each iteration $k$ one needs to choose whether to employ the fine or the coarse step. A sketch of a possible \name{} cycle of iterations is depicted in Figure \ref{fig:schema}.

\begin{figure}[h]
\centering
\includegraphics[width=0.8\textwidth]{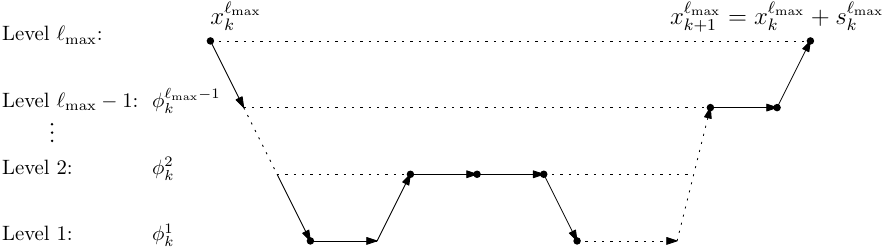}
\caption{Sketch of a possible iteration scheme for \name{}. Horizontal arrows represent fine steps.} 
 \label{fig:schema}
\end{figure}

 \paragraph{The step acceptance}
 The step $s_k^\ell$ obtained at each level is used to define a trial point $x_k^\ell+s_k^\ell$
 and an estimate $f_k^{\ell,s}$ of $m_k^{R,\ell}(\xkl+\skl)$.
 The step acceptance is based on the ratio of the achieved reduction over the predicted reduction:
\begin{equation}\label{rho}
\rho_k=\frac{f_k^{\ell}-f_k^{\ell,s}}{m_{k}^\ell(0)-m_{k}^\ell(\skl)},
\end{equation}
where $f_k^{\ell}$ is an estimate of  $m_k^{R,\ell}(\xkl)$, and for the fine step  $m_{k}^{\ell}-m_{k}^{\ell,s}=-(g_k^\ell)^Ts_k^\ell$ with $g_k^\ell$ an estimate of $\nabla_s m_k^{R,\ell}(\xkl)$ and for the coarse step $m_{k}^{\ell}-m_{k}^{\ell,s}=\varphi_k^{\ell-1}(0)-\varphi_k^{\ell-1}(s_k^{\ell-1})$.   
 A \textit{successful} iteration is declared if the model is accurate, i.e., $\rho_k$ is larger than or equal to a chosen threshold $\eta_1\in(0,1)$
and $\|g_k^\ell\|\geq \frac{\eta_2}{\lambda_k^\ell}$ for some $\eta_2 >0$; otherwise the iteration is declared \textit{unsuccessful} and the step is rejected. The test for the step acceptance is combined with the update
of the regularization parameter $\lambda_k^\ell$ for the next iteration. The update is still based on the ratio \eqref{rho}. If the step is successful, the regularization parameter is decreased, otherwise it is increased.

The full multilevel procedure with $\ell$ levels, specialized for finite sum minimization problems, is described in Algorithm \ref{alg2: STREG} and will be introduced in section \ref{sec:multilevel}.  In the following section, for sake of simplicity, we detail the procedure in the two-level case.

\subsection{\namedue{}: the two-level case }\label{sec:2leva}
 We assume here that we have just two levels, i.e., just a coarse approximation to $f$. We therefore omit the superscript $\ell$ and we denote by $\coarse_k$ the coarse approximation at iteration $k$. Moreover, let $n_c\leq n$ be the dimension of the coarse space,  and let $R_k$ and $P_k$ be the grid operators.	We sketch the \name{} procedure for $\ell=2$ in Algorithm \ref{alg:STGM} where we rename it as \namedue. In the following, we collect the main assumptions on the algorithmic steps that will be used in the convergence analysis in the next section.

\begin{assumption}\label{hp:passi}
  At each iteration $k$ of Algorithm \ref{alg:STGM}, given $f_k\in\mathbb{R}$ and $g_k\in\mathbb{R}^n$, approximations to $f(x_k)$ and $\nabla_x f(x_k)$ respectively,  let the step $s_k\in\mathbb{R}^{n}$ be computed so that either:
\begin{align}
&s_k=-\frac{\gk}{\lambda_k\|\gk\|}, & \text{(fine step)} \text{ or }\\
& s_k=P_k s^*, \; s^*\in\mathbb{R}^{n_c}, & \text{(coarse step)}
\end{align}
with $s^*$ satisfying
\begin{equation}\label{stopping}
     m_k^{R}(s^*)\leq m_k^{R}(0) \ \text{ and }\ \|\nabla_s m_k^{R}(s^*)\|
          \leq \theta \|s^*\|, \; \theta>0  
 \end{equation}
where, for $s\in\mathbb{R}^{n_c}$
\begin{subequations}\label{reg_taylor_model}
    \begin{align}
\varphi_k(s)&:=  \coarse_k(R_kx_k+s)+(R_k\gk-\nabla_s\coarse_k(R_kx_k))^Ts,\\
	m_k^{R}(s)&:= \varphi_k(s)+
	  \frac{\lambda_k\|\gk\|}{2}\|s\|^{2}.
    \end{align}
    \end{subequations}
    The definition of the coarse model ensures  that
\begin{equation}\label{coherence}
\nabla_s\varphi_k(0)=R_k\gk
\end{equation}
and that, when the coarse model is used, it holds:
\begin{equation}\label{decr_varphi}
  \varphi_k(s^*)-\varphi_k(0)\leq -\frac{1}{2}\regc\|s^*\|^2.
\end{equation}
The use of the coarse step is restricted to iterations $k$ such that 
    \begin{equation}\label{go_down_condition}
    \| R_k \gk\|\geq \kappa_H \|\gk\|
\end{equation}
for $\kappa_H\in(0,\min_k\min\{1,\|R_k\|\})$. We assume that $R_k=\nu_k P_k^T$ with 
 $\nu_k= 1$ for all $k$, without loss of generality, and that, for all $k$, $\|R_k\|\leq \kappa_R$ with $\kappa_R>0$. 
 This in particular, from \eqref{coherence}, ensures that, if $s_k=P_ks^*$ then 
 \begin{equation}\label{eq:good_direction}
     \nabla_s\varphi_k(0)^Ts^*=\gk^Ts_k.
 \end{equation}

\end{assumption}

\begin{algorithm}{}
  \caption{\namedue$(x_{0},\lambda_{0})$ two-levels stochastic regularized gradient method }\label{alg:STGM} 
  	\begin{algorithmic}[1]
  \State \label{init} $\bullet$ {\bf Initialization:} Choose $x_{0} \in \mathbb{R}^{n}$ and $\lambda_{0} >0$. Set the constants 	$\eta_1\in (0,1)$, $\eta_2 > 0$ and $\gamma \in (0,1)$. Set $k=0$. 
   \State \label{init2} $\bullet$ {\bf Operators and functions computation:} Choose a (possibly random) restriction operator $R_k$ and compute $g_k$, a fine approximation of $\nabla_x f(x_k)$.
\State \label{model_choice} $\bullet$ {\bf Model choice:} If \eqref{go_down_condition} holds, choose if to use the fine level model and go to Step \ref{finestep}, or the lower level model and go to Step \ref{coarsestep}. Otherwise go to Step \ref{finestep}. 
 
\State\label{finestep} $\bullet$ {\bf Fine step computation:} Build a model $m_{k}(x_k+s)=f_k+ \gk^Ts$ that approximates $f(x)$ around $x_k$. Set $s_k=-\frac{\gk}{\lambda_k\|\gk\|}$, $m_k^0=m_k(x_k)$ and $m_k^s=m_k(x_k+s_k)$. Go to Step \ref{step_acceptance}.
  
\State\label{coarsestep} $\bullet$ {\bf Coarse step computation:}  Randomly define a coarse approximation $\coarse_k$ of $f$, and the lower level model and its regularized version as:
\begin{align*}
		\varphi_k(s)&=\coarse_k(R_k\, x_k+s)+[R_k\,\gk -\nabla_x \coarse_k(R_k\,x_{k})]^Ts,\\
  m_k^R(s)&=\varphi_k(s)+ \frac{1}{2} \regc\|s\|^2.
		\end{align*}
		 Approximately minimize $m_k^R$, yielding an approximate solution $s^*$ satisfying \eqref{stopping}. Define $s_k=P_k s^*$ and set $m_k^0=\varphi_k(0)$ and $m_k^s=\varphi_k(s^*)$.
  
\State $\bullet$ \label{step_acceptance} {\bf Acceptance of the trial point and regularization parameter update:} 
  Obtain an estimate $f_k^s$ of $f(x_k+s_k)$ and compute $\rho_k=\displaystyle \frac{f_k-f_k^s}{m_{k}^0-m_{k}^s}.$ 
 \Statex {\bf If} $\rho_k\geq \eta_1$ and $\|\gk\|\geq \eta_2/\lambda_k$ {\bf then} set  $x_{k+1}=x_k+s_k$ and $\lambda_{k+1}= \gamma\lambda_k$.
 \Statex{\bf Else} set $x_{k+1}=x_k$ and $\lambda_{k+1}= \gamma^{-1}\lambda_k$.
 \State $\bullet$ {\bf Check stopping criterion}. If satisfied stop, otherwise set $k = k+1$ and go to Step \ref{init2}.
   \end{algorithmic}
\end{algorithm}

\section{Convergence theory}\label{sec_conv}

In this section, we provide a theoretical analysis of the proposed multilevel method proving the almost sure global convergence to first order critical points.
Note that, as the method is recursive, we can restrict the analysis to the two-levels case. 
We thus focus on \namedue{} as described in section \ref{sec:2leva}.
The analysis follows the scheme proposed in \cite{storm} and is extended to adaptive regularization methods while including the multilevel steps. 
Let us now first state  some regularity assumptions as in \cite{birgin2017worst}. 
\begin{assumption}\label{hp_lip_grad}
	Let $f:\mathbb{R}^n\rightarrow\mathbb{R}$ and $\coarse_k:\mathbb{R}^{n_c}\rightarrow\mathbb{R}$ with $n\geq n_c$, be continuously differentiable and bounded from below functions, for all $k$.
	Let us assume that the gradients of $f$ and of  $\coarse_k$  are Lipschitz continuous, i.e., that there exist constants $L_f$ and $L_{\coarse_k}$ such that 
	\begin{align*}
 \|\nabla_x f (x)-\nabla_x f(y)\| &\leq L_f\, \|x-y\|\quad \text{ for all} \quad x,y\in\mathbb{R}^{n}, \\
	\|\nabla_x\coarse_k(x)-\nabla_x\coarse_k(y)\| &\leq L_{\coarse_k}\, \|x-y\|\quad \text{ for all} \quad x,y\in\mathbb{R}^{n_c}.
	\end{align*}
    Let $L_\phi:=\max_k L_{\phi_k}$.
\end{assumption}
\begin{remark}
    In the setting of finite sum minimization, cf. Example 1, the assumption is satisfied if all the $f^{(i)}$ are smooth, which is quite a common assumption in this context \cite{garrigos2023handbook}. 
\end{remark}
Moreover, we assume that the models in this work are random functions and so is their behavior and influence on the iterations. Hence, $M_k$ will denote a random model in the $k$-th
iteration, while we will use the notation $m_k = M_k(\omega)$ for its realizations. As a consequence of
using random models, the iterates $X_k$, the regularization parameter $\Lambda_k$ and the steps $S_k$ are also random
quantities, and so $x_k = X_k(\omega)$, $\lambda_k = \Lambda_k(\omega)$, $s_k = S_k(\omega)$ will denote their respective realizations.
Similarly, let random quantities $F_k, F_k^s$ 
denote the estimates of $f (X_k)$ and $f (X_k + S_k)$, with
their realizations denoted by $f_k=F_k
(\omega)$ and $f^s_k=F^s_k(\omega)$. In other words, Algorithm \ref{alg:STGM} results in
a stochastic process $\{M_k, X_k, S_k, \Lambda_k, F_k, F^s_k\}$. Our goal is to show that under certain conditions
on the sequences $\{M_k\}$ and $\{F_k,F^s_k\}$
 the resulting stochastic process has desirable convergence
properties with probability one. 
In particular, we will assume that models $M_k$ and estimates $F_k,F_k^s$ are sufficiently accurate with sufficiently high probability, conditioned on the past.
To formalize conditioning on the past, let $\mathcal{F}^{M \cdot F}_{k-1}$ denote the $\sigma$-algebra generated by $M_0,\dots,M_{k-1}$ and $F_0, \dots, F_{k-1}$ and let $\mathcal{F}^{M \cdot F}_{k-1/2}$ denote the $\sigma$-algebra generated by $M_0,\dots,M_{k}$ and $F_0, \dots, F_{k-1}$. 
To formalize sufficient accuracy we use the  measure for the accuracy 
introduced in \cite{bergou22}, which adapts to regularized models those originally proposed in \cite{storm}.

\begin{definition}\label{def:fullylin}
Suppose that $\nabla f$ is Lipschitz continuous. Given $\lambda_k>0$, a function $m$ is a $\kappa$-fully linear model of
$f$ around the iterate $x_k$ provided, for $\kappa = (\kappa_f , \kappa_g)$, that for all $y$ in a neighbourhood of $x_k$:
\begin{align}
\|\nabla_x f (y) - \nabla_x m(y)\|\leq \frac{\kappa_g}{\lambda_k}, \label{kling}\\
|f (y) - m(y)| \leq \frac{\kappa_f}{\lambda_k^2}.\label{klinf}
\end{align}

\end{definition}

We will ask for this requirement on the fine level model $m_k(x_k+s)=f_k+g_ks^T$. 
Specifically, we will consider probabilistically fully-linear models, according to the following definition \cite{storm}:

\begin{definition} 
A sequence of random models $\{M_k\}$ is said to be $\alpha$-probabilistically $\kappa$-fully linear
with respect to the corresponding sequence $\{X_k,\Lambda_k\}$ 
if the events
$$
I_k = \{M_k \text{ is a } \kappa\text{ -fully linear model of } f \text{ around } X_k\} 
$$
satisfy the condition
$$
\mathbb{P}(I_k|\mathcal{F}^{M\, F}_{k-1}) \geq \alpha,
$$
where $\mathcal{F}^{M\, F}_{k-1}$
is the $\sigma$-algebra generated by $M_0,\dots,M_{k-1}$
and $F_0,\dots,F_{k-1}$.
\end{definition}

Notice that imposing this condition  on the fine level only will be enough to ensure convergence of the method. The first order correction imposed on the coarse model will ensure a link between the two approximations of $f$ and consequently  the coarse step will point in the good direction thanks to the  link with the fine one\footnote{If $R_k$ is the identity, the coherence term makes the coarse models fully linear in a neighbourhood of  $x_k$.  }, cf. \eqref{eq:good_direction}.  Thus the accuracy of the coarse approximations does not need to increase along with the iterations, or the increase can be much slower than at fine level.

We will also require function estimates to be sufficiently accurate. 

\begin{definition}\label{def:eps_acc}
  The estimates $f_k$ and $f_k^s$ are said to be $\epsilon_f$-accurate estimates of $f(x_k)$ and $f(x_k+s_k)$ respectively, for a given $\lambda_k$ if 
  $$
  \lvert f_k-f(x_k)\rvert\leq \frac{\epsilon_f}{\lambda_k^2} \text{ and } \lvert f_k^s-f(x_k+s_k)\rvert\leq \frac{\epsilon_f}{\lambda_k^2}.
  $$
\end{definition}

In particular we will consider probabilistically accurate estimates as in \cite{storm}:

\begin{definition} 
A sequence of random estimates $\{F_k,F_k^s\}$ is said to be $\beta$-probabilistically $\epsilon_f$-accurate
with respect to the corresponding sequence $\{X_k,\Lambda_k,S_k\}$
if the events
$$
J_k = \{F_k,F_k^s \text{ are } \epsilon_f\text{-accurate estimates of } f(x_k) \text{ and } f(x_k+s_k), \text{ respectively, around } X_k\} 
$$
satisfy the condition
$$
\mathbb{P}(J_k|\mathcal{F}^{M\, F}_{k-1/2}) \geq \beta,
$$
where $\epsilon_f$ is a fixed constant and $\mathcal{F}^{M\, F}_{k-1/2}$
is the $\sigma$-algebra generated by $M_0,\dots,M_{k}$
and $F_0,\dots,F_{k-1}$.
\end{definition}

Following \cite{storm}, in our analysis we will require that our method has access 
to $\alpha$-probabilistically $\kappa$-fully linear models, for some fixed $\kappa$ and to $\beta$-probabilistically $\epsilon_f$ accurate function estimates, for some fixed, sufficiently small $\epsilon_f$. Cf. \cite[Section 5]{storm} 
for procedures for constructing probabilistically fully linear models, and probabilistically accurate estimates. Basically, when the function approximations come from a subsampling this construction is possible if the model accounts for enough samples.

\subsection{Convergence analysis}

We start by recalling two useful relations, following from Taylor's theorem, see for example \cite[Corollary A.8.4]{book_compl}.
\begin{lemma}
	 Let $h:\mathbb{R}^n\rightarrow\mathbb{R}$ be a continuously differentiable function with Lipschitz continuous gradient, with $L$ the corresponding Lipschitz constant. Given its first order truncated Taylor series in $x$, $T[h](s) := h(x)+\nabla_x h(x)^Ts$, it holds:
\begin{align}
& h(x+s)= T[h](s)+\int_{0}^{1}[\nabla_x h(x+\xi s)-\nabla_x h(x)]^T{s}\,d\xi,\label{taylor1}\\
&\lvert h(x+s)- T[h](s)\rvert \leq \frac{L}{2}\|s\|^{2},\label{taylor2}
\end{align} 
\end{lemma}
We now propose two technical lemmas on the coarse step. 

\begin{lemma}\label{lemma_rho_new}
Let Assumptions \ref{hp:passi} and \ref{hp_lip_grad} hold.
 Consider a realization of Algorithm \ref{alg:STGM} where at iteration $k$ the coarse model is used and let $s_k=Ps^*$ be the resulting step. Then it holds: 
\begin{equation}\label{taylor_varphi}
\lvert \varphi_k(0)+\varphi_k(s^*)-g_k^Ts_k\rvert\leq \frac{L_{\coarse}}{2}\|s^*\|^2.
\end{equation}
\end{lemma}

\begin{proof}
Using the first order Taylor expansion of $\varphi_k$ and \eqref{taylor1} applied to $\varphi_k$, and considering that from \eqref{coherence}, $\nabla_s \varphi_k(0)^Ts^*=\gk^Ts_k$,
we can write:
\begin{align*}
\varphi_k(0)-\varphi_k(s^*)= 
-\gk^Ts_k
-\int_{0}^{1} \left[\nabla_s \varphi_k(\xi s^*)-\nabla_s \varphi_k(0)\right]^Ts^* \;d\xi.
 \end{align*}
Using Assumption \ref{hp_lip_grad} and recalling that $\varphi_k$ and $\coarse_k$ just differ  by a linear term, we obtain:
\begin{align*}
\lvert \varphi_k(0)-\varphi_k(s^*)+g_k^Ts_k)-\rvert&
\leq\int_{0}^{1}\lvert \left[\nabla_s \varphi_k(\xi s^*)-\nabla_s \varphi(0)\right]^Ts^* \rvert\;d\xi\\
&\leq \int_{0}^{1} \|\nabla_s \varphi_k(\xi s^*)-\nabla_s \varphi_k(0)\| \|s^*\|\; d\xi
 \leq \frac{L_{\coarse}}{2}\|s^*\|^2.
\end{align*}

\end{proof}

\begin{lemma}\label{lemma_grad}
  Under Assumptions \ref{hp:passi} and \ref{hp_lip_grad}, for any realization of Algorithm \ref{alg:STGM} and for each iteration $k$ where the coarse step is used it holds:
   \begin{equation}\label{rel_g_s}
  \|R g_k\|\leq (L_\phi+\theta+\regc) \|s^*\|.
  \end{equation}
\end{lemma}

\begin{proof} From \eqref{stopping} and \eqref{reg_taylor_model}, since $\nabla_sm_k^R(s)=\nabla_s\varphi_{k}(s)+\lambda_k\|\gk\|s$, it follows:
\begin{align*}
\|R \gk  \|& \leq \|R\gk-\nabla_s\varphi_{k}(s^*)\|+\|\nabla_s\varphi_{k}(s^*)+\lambda_k\|\gk\|s^*\|+ \lambda_k\|\gk\|\|s^*\|\\
&\leq \|R\gk-\nabla_s\phi_k(R_kx_k+s^*)-Rg_k+\nabla_s \phi_k (Rx_k)\|+\theta \|s^*\|+\lambda_k\|\gk\|\|s^*\|\\
&\leq L_\phi \|s^*\|+\theta \|s^*\|+\lambda_k\|\gk\|\|s^*\|.
\end{align*}
Thus we finally obtain
the thesis. 
\end{proof}

The following lemma relates the coarse step size and the regularization parameter $\lambda_k$. 
\begin{lemma}
Let Assumptions \ref{hp:passi} and \ref{hp_lip_grad} hold.
Assume that at iteration $k$ the coarse step is used. 
 Assume that 
  \begin{equation}\label{hplc}
    \frac{1}{\lambda_k}\leq \min \Big\{\frac{1}{L_\phi+\theta},\frac{1}{2L_{\coarse}}\Big\}\|\gk\|,
  \end{equation}
  
  Then 
\begin{equation}\label{rel_g_s2}
    \frac{\kappa_H}{2\lambda_k}\leq \|s^*\|\leq \frac{4\kappa_R}{\lambda_k}.
  \end{equation}

\end{lemma}

\begin{proof}
The first inequality follows from:
\begin{equation}
   \|s^*\|\overset{\eqref{rel_g_s}}{\geq} \frac{\|R_kg_k\|}{ L_\phi+\theta+\regc} \overset{\eqref{go_down_condition}}{\geq} \frac{\kappa_H\|g_k\|}{L_\phi+\theta+\regc} \overset{\eqref{hplc}}{\geq} \frac{\kappa_H}{2\lambda_k}.
  \end{equation}
The second inequality follows from \eqref{taylor_varphi}: 
$$
|\varphi_k(s^*)-\varphi_k(0)|-|\nabla_s\varphi_k(0)^Ts^*|\leq|\varphi_k(s^*)-\varphi_k(0)-\nabla_s\varphi_k(0)^Ts^*|\leq \frac{L_{\coarse}}{2}\|s^*\|^2,
$$
where we have used the fact that from \eqref{reg_taylor_model} and Assumption \ref{hp_lip_grad}, $\varphi_k$ is $L_{\coarse}$-smooth. 
Thus, from \eqref{eq:good_direction} and since $\varphi_k(0)\geq \varphi_k(s^*)$, we obtain
\begin{align*}
\varphi_k(0)-\varphi_k(s^*)&\leq \lvert \nabla_s\varphi_k(0)^Ts^*\rvert +\frac{L_{\coarse}}{2}\|s^*\|^2=\lvert \gk^Ts_k\rvert +\frac{L_{\coarse}}{2}\|s^*\|^2\\
&\leq \|\gk\|\|s_k\|+\frac{L_{\coarse}}{2}\|s^*\|^2\leq \kappa_R\|\gk\|\|s^*\|+\frac{L_{\coarse}}{2}\|s^*\|^2.
\end{align*}
Combining this with  \eqref{decr_varphi} we have: 
$$
\frac{1}{2}\regc \|s^*\|^2\leq \varphi_k(0)-\varphi_k(s^*)\leq  \kappa_R\|\gk\|\|s^*\|+\frac{L_{\coarse}}{2}\|s^*\|^2.
$$
Thus
$$
\left(\frac{1}{2}\regc -\frac{L_{\coarse}}{2}\right)\|s^*\|^2\leq \kappa_R\|\gk\|\|s^*\|.
$$
From \eqref{hplc} we have $\frac{1}{2} \regc- \frac{L_{\coarse}}{2}\geq\frac{1}{4}\regc$ and 
thus 
$$
\frac{1}{4}\regc \|s^*\|\leq  \kappa_R\|\gk\|.
$$
\end{proof}

In the following lemma we measure the decrease predicted by the model. 
\begin{lemma}
  Let Assumptions \ref{hp:passi} and \ref{hp_lip_grad}  hold.
  For any realization of Algorithm \ref{alg:STGM} and for each $k$ it holds:
  \begin{equation}\label{diff_mod}
     m_{k}^s-m_{k}^0\leq\begin{cases} -\frac{\|\gk\|}{\lambda_k} & \text{ if fine step,}\\
         -\frac{\regc}{2}\|s^*\|^2& \text{ if coarse step}.
        \end{cases}
    \end{equation}

\end{lemma}
\begin{proof}
If the fine step is used,
        \begin{align*}
    m_{k}^s-m_{k}^0
           =\gk^Ts_k=-\frac{\|\gk\|^2}{\lambda_k\|\gk\|}=-\frac{\|\gk\|}{\lambda_k}.
        \end{align*}
If the coarse step is used:
         \begin{align*}
           m_{k}^s-m_{k}^0&= 
        \varphi_k(s^*)-\varphi_k(0)\overset{\eqref{decr_varphi}}{\leq} -\frac{\regc}{2}\|s^*\|^2. 
        \end{align*}
        
\end{proof}

We now prove some auxiliary lemmas that provide conditions under which the decrease of the true objective function $f$ is guaranteed. The first lemma states that if the regularization parameter is large enough relative to the size of the fine model gradient and if the fine model is fully linear, then the step $s_k$ provides a decrease in $f$ proportional to the size of the model gradient.

\begin{lemma}\label{lemma4.5_STORM} Under Assumptions  \ref{hp:passi} and \ref{hp_lip_grad},
suppose that $m_k(x_k+s):=f_k+g_k^Ts$ is a $(\kappa_f,\kappa_{g})$-fully linear model of $f$ in a neighbourhood of $x_k$. If 
\begin{equation}\label{hpl}
    \frac{1}{\lambda_k}\leq \min\Bigg\{\frac{1}{L_{\coarse}+\theta},\frac{\kappa_H^2}{64\kappa_f},\frac{1}{ 2L_{\coarse}}\Bigg\}\|\gk\|,
\end{equation} 
then the trial step $s_k$ leads to an improvement in $f(x_k+s_k)$ such that
$$
f(x_k+s_k)-f(x_k)\leq -\frac{\kappa_H^2}{32}\frac{\|\gk\|}{\lambda_k}.
$$
\end{lemma}

\begin{proof}

We distinguish two cases depending on the used step.
    \begin{enumerate}
        \item In the fine step case,   we get
        \begin{align*}
    f(x_k+s_k)-f(x_k)&=f(x_k+s_k)-m_k(x_k+s_k)+m_k(x_k+s_k)-m_k(x_k)+m_k(x_k)-f(x_k)\\
    &\overset{\eqref{klinf}+\eqref{diff_mod}}{\leq} \frac{2\kappa_f}{\lambda_k^2}-\frac{\|\gk\|}{\lambda_k}\overset{\eqref{hpl}}{\leq} -\frac{1}{2}\frac{\|\gk\|}{\lambda_k}\leq  -\frac{\kappa_H^2}{32}\frac{\|\gk\|}{\lambda_k},
\end{align*}
where we have used that, from \eqref{hpl}, $\frac{1}{\lambda_k}\leq \frac{\kappa_H^2}{64\kappa_f}\|\gk\|\leq \frac{1}{4\kappa_f}\|\gk\|$, since $\kappa_H\leq 1$. 
\item When the coarse step is used, we have, for $m_k$ the fine model, that
\begin{align*}
    f(x_k+s_k)-f(x_k)&=f(x_k+s_k)-m_k(x_k+s_k)\\
    &+m_k(x_k+s_k)-m_k(x_k)-\varphi_k(s^*)+\varphi_k(0)\\
    &-\varphi_k(0)+\varphi_k(s^*)\\
    &+m_k(x_k)-f(x_k).
\end{align*}
The first and the last terms are bounded by $\frac{\kappa_f}{\lambda_k^2}$ from \eqref{klinf}. The second term from Lemma \ref{lemma_rho_new} is bounded by $\frac{L_{\coarse}}{2}\|s^*\|^2$. The third term is bounded by $-\frac{{\regc}}{2}\|s^*\|^2$ from \eqref{decr_varphi}.
Thus
    \begin{align*}
    f(x_k+s_k)-f(x_k)&\leq 
        \frac{2\kappa_f}{\lambda_k^2}+\left(\frac{L_{\coarse}}{2}-\frac{\lambda_k\|\gk\|}{2}\right)\|s^*\|^2\\
    &\overset{\eqref{hpl}}{\leq}
\frac{2\kappa_f}{\lambda_k^2}-\frac{\regc}{4}\|s^*\|^2\\
      &\overset{\eqref{rel_g_s2}}{\leq}\frac{2\kappa_f}{\lambda_k^2}-\frac{\|\gk\|\kappa_H^2}{16\lambda_k}
      \\&\overset{\eqref{hpl}}{\leq} -\frac{\kappa_H^2}{32}\frac{\|\gk\|}{\lambda_k}.
     \end{align*}
    
    \end{enumerate}

\end{proof}

The next lemma shows that for a sufficiently large regularization parameter $\lambda_k$ relative to the size of the true gradient $\nabla_x f(x_k)$, the guaranteed decrease in the objective function, provided by $s_k$, is proportional to the size of the true gradient.

\begin{lemma} \label{lem:46storm}
Let Assumptions \ref{hp:passi} and \ref{hp_lip_grad} hold and 
suppose that $m_k$ is a $(\kappa_f,\kappa_{g})$-fully linear model of $f$ in a neighbourhood of $x_k$. If 
\begin{equation}\label{hpl2}
    \frac{1}{\lambda_k}\leq \min\Bigg\{\frac{1}{L_\phi+\theta+\kappa_g},\frac{1}{(64\kappa_f/\kappa_H^2)+\kappa_g}, \frac{1}{ 2L_{\coarse}+\kappa_g}\Bigg\}\|\gradf\|,
\end{equation}
then the trial step $s_k$ leads to an improvement in $f(x_k+s_k)$ such that
$$
f(x_k+s_k)-f(x_k)\leq -C_1\frac{\|\gradf\|}{\lambda_k},
$$
with $C_1:=\frac{\kappa_H^2}{32}\max\Big\{\frac{L_{\phi}+\theta}{L_{\phi}+\theta+\kappa_g},\frac{64\kappa_f}{64\kappa_f+\kappa_g\kappa_H^2},\frac{{2L_{\coarse}}}{{2L_{\coarse}}+\kappa_g}\Big\}$.
\end{lemma}
\begin{proof}
We first prove that the assumption of Lemma \ref{lemma4.5_STORM} is satisfied, and we use its result to deduce the decrease of the objective function in terms of $\|\gradf\|$ rather than $\|\gk\|$, by linking these two quantities through the assumption of $\kappa$-fully linear model, which yields that
\begin{equation}\label{flin}
           \|\gk\|\geq \|\gradf\|-\frac{\kappa_g}{\lambda_k}.
\end{equation}
From assumption \eqref{hpl2} it holds 
    \begin{equation*}
        \|\gradf\|\geq  \max\{L_{\phi}+\theta+\kappa_g,64\kappa_f/\kappa_H^2+\kappa_g,{2 L_{\coarse}}+\kappa_g \}\frac{1}{\lambda_k},
    \end{equation*}
and thus from \eqref{flin} we have 
    $$
    \|\gk\|\geq \|\gradf\|-\frac{\kappa_g}{\lambda_k}\geq\max\{L_{\phi}+\theta,64\kappa_f/\kappa_H^2, {2 L_{\coarse}}\}\frac{1}{\lambda_k}.
    $$
    Thus the assumption of Lemma \ref{lemma4.5_STORM} is satisfied and 
    $$
    f(x_k+s_k)-f(x_k)\leq  -\frac{\kappa_H^2}{32}\frac{\|\gk\|}{\lambda_k}. 
    $$
In the same way from \eqref{hpl2} and \eqref{flin} we have
    \begin{align*}
    \|\gk\|&\geq \|\gradf\|-\frac{\kappa_g}{\lambda_k}\\
    &\geq  \|\gradf\|-\kappa_g\min\Big\{\frac{1}{L_{\phi}+\theta+\kappa_g},\frac{1}{64\kappa_f/\kappa_H^2+\kappa_g}, \frac{1}{2L_{\coarse}+\kappa_g}\Big\}\|\gradf\|\\
    &=\max\Big\{\frac{L_{\phi}+\theta}{L_{\phi}+\theta+\kappa_g},\frac{64\kappa_f}{64\kappa_f+\kappa_g\kappa_H^2}, \frac{{2L_{\coarse}}}{{2 L_{\coarse}}+\kappa_g}\Big\}\|\gradf\|:=\tilde{C}_1\|\gradf\|.
    \end{align*}
We conclude that 
    $$
    f(x_k+s_k)-f(x_k)\leq  -\frac{{\kappa_H^2}}{32}\frac{\|\gk\|}{\lambda_k}\leq -\frac{{\kappa_H^2}\tilde{C}_1}{32}\frac{\|\gradf\|}{\lambda_k}:=-C_1\frac{\|\gradf\|}{\lambda_k}.
    $$

\end{proof}

We now prove a lemma that states that, if the estimates are sufficiently accurate, the fine model is fully-linear and the regularization parameter is large enough relatively to the size of the model gradient, then a successful step is guaranteed. 

\begin{lemma}
Let Assumptions \ref{hp:passi} and \ref{hp_lip_grad}  hold. Suppose that $m_k$ is a $(\kappa_f,\kappa_g)$-fully linear model in a neighbourhood of $x_k$ and that the estimates $\{f_k,f_k^s\}$ are $\epsilon_f$-accurate with $\epsilon_f\leq \kappa_f$. If 
\begin{equation}\label{hp}
    \frac{1}{\lambda_k}\leq \min\Big\{\frac{1}{L_{\coarse}+\theta},\frac{1}{\eta_2},\frac{1-\eta_1}{32\kappa_f/\kappa_H^2+ L_{\coarse}}\Big\}\|\gk\|,
\end{equation}
then the $k$-th iteration is successful.

\end{lemma}

\begin{proof}
Let us consider $\rho_k$ in Step \ref{step_acceptance} of Algorithm \ref{alg:STGM}. When the fine step is used, it holds:
 \begin{align}
        \rho_k&=\frac{f_k-f_k^s}{m_{k}^0-m_{k}^s}
        =\frac{m_k^0-m_k^s}{m_k^0-m_k^s}
        +\frac{m_k^s-f(x_k+s_k)}{m_k^0-m_k^s}
        +\frac{f(x_k+s_k)-f_k^s}{m_k^0-m_k^s}\label{develop_rho}.
 \end{align}
 From the assumption on the function estimates (cf. Definition \ref{def:eps_acc}) it follows:
 \begin{align*}
  \lvert f_k^s-f(x_k+s_k)\rvert\leq \frac{\epsilon_f}{\lambda_k^2}\leq \frac{\kappa_f}{\lambda_k^2}.
  \end{align*}
From the assumption of $\kappa$-fully linearity of the model the numerator of the second term is bounded by \eqref{klinf}.
  Consequently, the numerator of $|\rho_k-1|$ is bounded by
  $
  \frac{2\kappa_f}{\lambda_k^2}.
  $
The denominator is bounded from \eqref{diff_mod}. 
  Thus by \eqref{hp} 
  $$
  \lvert\rho_k-1\rvert \leq  \frac{2\kappa_f} {\lambda_k\|\gk\|}\leq 1-\eta_1.
  $$
If the coarse step is used we have: 
\begin{align*}
        \rho_k&=\frac{f_k-f_k^s}{m_k^0-m_k^s}\\
        &=\frac{f_k+g_k^Ts_k-f(x_k+s_k)}{\varphi_k(0)-\varphi_k(s^*)}
        +\frac{\varphi_k(s^*)-\varphi_k(0)-g_k^Ts_k}{\varphi_k(0)-\varphi_k(s^*)}
        +\frac{\varphi_k(0)-\varphi_k(s^*)}{\varphi_k(0)-\varphi_k(s^*)}\\
       & +\frac{f(x_k+s_k)-f_k^s}{\varphi_k(0)-\varphi_k(s^*)}.
 \end{align*}
Considering the first and the last term,  the numerators can be bounded as in the first case. We thus have 
  \begin{align*}
 \Bigg\lvert \frac{f_k+g_k^Ts_k-f(x_k+s_k)}{\varphi_k(0)-\varphi_k(s^*)}  +\frac{f(x_k+s_k)-f_k^s}{\varphi_k(0)-\varphi_k(s^*)} \Bigg\rvert\overset{\eqref{decr_varphi}}{\leq}\frac{
  \frac{2\kappa_f}{\lambda_k^2}
  }{
\frac{\regc}{2}\|s^*\|^2
}\overset{\eqref{rel_g_s2}}{\leq}    
\frac{
  \frac{2\kappa_f}{\lambda_k^2}
  }{
\frac{\regc}{2}\frac{\kappa_H^2}{4\lambda_k^2}
}
=\frac{32\kappa_f}{\lambda_k\|\gk\|\kappa_H^2}.
  \end{align*}
For the second term we have:
 \begin{align*}
  \Bigg\lvert\frac{\varphi_k(s^*)-\varphi_k(0)-g_k^Ts_k}{m_k^0-m_k^s}\Bigg\rvert\overset{\eqref{taylor_varphi}+\eqref{decr_varphi}} {\leq} \frac{\frac{L_{\coarse}}{2}\|s^*\|^2}{\frac{\regc}{2}\|s^*\|^2}=\frac{L_{\coarse}}{\regc}.
  \end{align*}
 Thus from \eqref{hp} 
 
  \begin{align*}
  \lvert\rho_k-1\rvert &\leq  \frac{32\kappa_f/\kappa_H^2+L_{\coarse}} {\lambda_k\|\gk\|}   \leq 1-\eta_1.
  \end{align*}
  Hence in every case $\rho_k\geq 1$. Moreover, since $\|\gk\|\geq \frac{\eta_2}{\lambda_k}$ from \eqref{hp}, the $k$-th iteration is successful.
\end{proof} 

Finally, we state and prove the lemma that guarantees an amount of decrease of the objective function on a  true successful iteration. 

\begin{lemma} \label{lem:48storm}
Under Assumptions \ref{hp:passi} and \ref{hp_lip_grad} , suppose that the estimates $\{f_k,f_k^s\}$ are $\epsilon_f$-accurate with $\epsilon_f<\frac{\eta_1\eta_2\kappa_H^2}{16}$. If a trial step $s_k$ is accepted then the improvement in $f$ is bounded below by:
$$
f(x_{k+1})-f(x_k)\leq -\frac{C_2}{\lambda_k^2}
$$
where $C_2=\frac{\eta_1\eta_2\kappa_H^2}{8}-2\epsilon_f>0$.
\end{lemma}
\begin{proof}
If the iteration is successful, this means that $\|\gk\|\geq\frac{\eta_2}{\lambda_k}$ and $\rho_k\geq \eta_1$. Thus, if the fine step is used,
$$
f_k-f_k^s\geq \eta_1 (m_k^0-m_k^s)\overset{\eqref{diff_mod}}{\geq} \eta_1\frac{\|\gk\|}{\lambda_k}\geq\frac{\eta_1\eta_2}{\lambda_k^2}.
$$
If the coarse step is used
\begin{align*}
f_k-f_k^s&\geq \eta_1 (m_k^0-m_k^s)\overset{\eqref{diff_mod}}{\geq}
\frac{\eta_1}{2}\regc\|s^*\|^2\\
&\overset{\eqref{rel_g_s2}}{\geq}\frac{\eta_1\kappa_H^2}{8}{\|\gk\|}\frac{1}{\lambda_k}{\geq}\frac{\eta_1\eta_2\kappa_H^2}{8}\frac{1}{\lambda_k^2}.
\end{align*}
Then, 
since the estimates are $\epsilon_f$-accurate, we have that the improvement in $f$ can be bounded as 
$$
f(x_k+s_k)-f(x_k)=f(x_k+s_k)-f_k^s+f_k^s-f_k+f_k-f(x_k)\leq -\frac{C_2}{\lambda_k^2},
$$
where $C_2=\frac{\eta_1\eta_2\kappa_H^2}{8}-2\epsilon_f>0$.
\end{proof}

To prove convergence of Algorithm \ref{alg:STGM} we need to assume that the fine-level models $\{M_k\}$  and the estimates $\{F_k,F_k^s\}$ are sufficiently accurate with sufficiently high probability.

\begin{assumption}\label{hp:ab}
    Given values of $\alpha,\beta\in(0,1)$ and $\epsilon_f>0$, there exist $\kappa_g$ and $\kappa_f$ such that the sequence of fine-level models $\{M_k\}$ and estimates $\{F_k,F_k^s\}$ generated by Algorithm 1 are, respectively, $\alpha$-probabilistically $(\kappa_f,\kappa_g)$-fully-linear and $\beta$-probabilistically $\epsilon_f$-accurate. 
\end{assumption}

The following theorem states that the regularization parameter $\lambda_k$ converges to $+\infty$
 with probability one. Together with its corollary it gives conditions on the existence of $\kappa_g$ and $\kappa_f$ given $\alpha,\beta$ and $\epsilon_f$. 

 \begin{theorem}\label{thm}
     Let Assumptions \ref{hp:passi}, \ref{hp_lip_grad} and \ref{hp:ab}  be satisfied and assume that in Algorithm \ref{alg:STGM} the following holds.
     \begin{itemize}
         \item The step acceptance parameter $\eta_2$ is chosen so that 
         \begin{equation*}
             \eta_2\geq \max\{L_{\phi}+\theta,24\kappa_f\}.
         \end{equation*}
         \item The accuracy parameter of the estimates satisfies
         \begin{equation*}
            \epsilon_f\leq \min \left \{ \kappa_f, \frac{\eta_1\eta_2 {\kappa_H^2}}{32}\right\}.
         \end{equation*}
     \end{itemize}
     Then $\alpha$ and $\beta$ can be chosen so that, if Assumption \ref{hp:ab} holds for these values, then the sequence of regularization parameters $\{\Lambda_k\}$ generated by Algorithm \ref{alg:STGM} satisfies
     $$
\sum_{k=0}^{\infty}\frac{1}{\Lambda_k^2}<\infty
     $$
     almost surely.
 \end{theorem}

 \begin{proof}
     The scheme of the proof is the same as that of \cite[Theorem 4.11]{storm}. We outline here just the differences.  Let $C_1$  be defined as in Lemma \ref{lem:46storm}.
     We define
     $$
     h_k=\nu f(X_k)+(1-\nu)\frac{1}{\Lambda_k^2}
     $$
    with $\nu\in(0,1)$ such that 
     $$
     \frac{\nu}{1-\nu}>\max\Big\{\frac{4}{\gamma^2\zeta C_1},\frac{16}{\gamma^2\eta_1\eta_2 \kappa_H^2},\frac{1}{\gamma^23\kappa_f}\Big\} 
        $$
     where
     $$
     \zeta\geq \kappa_g+\max\Big\{\eta_2,\frac{64\kappa_f/\kappa_H^2+L_{\coarse}}{1-\eta_1}\Big\}.
     $$
     Under this assumption, the results in the proof of \cite[Theorem 4.11]{storm} hold with 
     \begin{align*}
         b_1&:=(1-\nu)(\gamma^2-1)\frac{1}{\lambda_k^2},\\
         b_2&:=-\nu C_1 \|\nabla_x f(x_k)\|\frac{1}{\lambda_k}+(1-\nu)\left(\frac{1}{\gamma^2}-1\right)\frac{1}{\lambda_k^2},\\
         b_3&:=\nu C_3 \|\nabla_x f(x_k)\|\frac{1}{\lambda_k}+(1-\nu)\left(\frac{1}{\gamma^2}-1\right)\frac{1}{\lambda_k^2},
     \end{align*}
     with $C_3:=\frac{40L_f}{\zeta}+4.$ In particular, there exists $\sigma>0$ such that for all $k$
     $$
     \E[h_{k+1}-h_k|\mathcal{F}_{k-1}^{M\, F}]\leq -\sigma{\frac{1}{\Lambda_k^2}}<0.
     $$
 \end{proof}
 The choice of $\alpha$ and $\beta$ is specified in the following corollary. 
 \begin{corollary}\label{cor}
 Let all assumptions of Theorem \ref{thm} hold. The statement of Theorem \ref{thm} holds if $\alpha$ and $\beta$ are chosen to satisfy the following conditions:
 $$
 \frac{\alpha\beta-\frac{1}{2}}{(1-\alpha)(1-\beta)}\geq\frac{\frac{40L_f}{\zeta}+4}{C_1}
 $$
 and
 $$
 (1-\alpha)(1-\beta)\leq \frac{\frac{1}{\gamma^2}-1}{\frac{1}{\gamma^4}-1+\frac{1}{\gamma^2}(40L_f+4\zeta)\max\Big\{\frac{4}{\zeta C_1},{\frac{16}{\eta_1\eta_2 \kappa_H^2}},\frac{1}{3\kappa_f}\Big\}},
 $$
 with 
$C_1=\frac{\kappa_H^2}{32}\max\Big\{\frac{L_{\phi}+\theta}{L_{\phi}+\theta+\kappa_g},\frac{64\kappa_f}{64\kappa_f+\kappa_g\kappa_H^2},\frac{{2L_{\coarse}}}{{2L_{\coarse}}+\kappa_g}\Big\}$ and $\zeta=\kappa_g+\eta_2$.
 \end{corollary}

 The following results can be derived as in  \cite[Lemma 4.17]{storm} and \cite[Theorem 4.18]{storm}, their proof is therefore omitted for sake of brevity.

 \begin{theorem}
     Let the assumptions of Theorem \ref{thm} hold. Let $\{X_k\}$ and $\{\Lambda_k\}$ be the sequences of random iterates and random regularization parameters generated by  Algorithm \ref{alg:STGM}. Fix $\epsilon>0$ and define the sequence $\{K_\epsilon\}$ consisting of the natural numbers $k$ for which $\|\nabla_xf(X_k)\|>\epsilon$. Then  almost surely 
     $$
     \sum_{k\in\{K_\epsilon\}}\frac{1}{\Lambda_k}<\infty.
     $$
 \end{theorem}

 \begin{theorem}
Let the assumptions of Theorem \ref{thm} hold. Let $\{X_k\}$ be the sequence of random iterates generated by Algorithm \ref{alg:STGM}. Then, almost surely,
$$
\lim_{k\rightarrow\infty}\|\nabla_x f(X_k)\|=0.
$$
 \end{theorem}

\section{\name{} for finite sum minimization }\label{sec:multilevel}

In this section we describe how to adapt Algorithm \ref{alg:STGM} to the solution of finite sum minimization problems of the form
\begin{equation}\label{pb_opti_fs}
\min_{x \in  \mathbb{R}^n} F(x) = \frac{1}{N}\sum_{i=1}^N f^{(i)}(x)
\end{equation}
 using a multilevel setting with $\ell$ levels. We assume that $N\gg n$ and we consider hierarchies built just in the samples space, thus at each level the iterates belong to $\mathbb{R}^n$ and the operators $R_k$ and $P_k$ are the identity. We thus drop here the indexes $\ell$ from the iterates and the steps.

Since the objective function in  \eqref{pb_opti_fs} is the average of the set of functions $\graffe{f^{(i)}}_{i=1}^N$, we can easily define a hierarchy of approximations by subsampling.
In particular, given the number of levels $\ell_{\max}\geq 2$, for every $\ell \in \graffe{1,...,\ell_{\max}}$ we define the subsampled function as:
\begin{equation}\label{coarse_model_l}
       \fsl(x) := \frac{1}{\lvert \Sub_k^\ell \lvert} \sum_{i\in \Sub_k^\ell}f^{(i)}(x);
\end{equation}
where $\Sub_k^\ell $ is a subsample set such that $\emptyset \neq \Sub_k^1 \subset ...\subset \Sub_k^\ell \subset...\subset \Sub_k^{\ell_{\max - 1}}\subset \Sub_k^{\ell_{\max}} \subseteq \graffe{1,...,N}$.

The stochastic framework of Algorithm \ref{alg:STGM} allows for inexact approximations of $F$ at the finest level, thus avoiding the need of the full sample evaluation at each iteration. Following \cite{storm}, in order to ensure that the model at fine level remains a fully linear model for $F$ we 
choose,  at iteration $k$, $\Sub_k^{\ell_{\max}}$ a set of randomly drawn samples  such that $\lvert\mathcal{S}_k^{\ell_{\max}}\rvert=p_k:=\min\{N,\max\{100k+n+2,\lambda_k^2\}\}$. The lower level approximations $\phi_k^{\ell}$ are defined as in \eqref{coarse_model_l} with $\Sub_k^\ell\subset \Sub_k^{\ell_{\max}}$ and $\lvert \Sub_k^\ell\rvert$ a fixed fraction\footnote{Note that they could as well be chosen constant.} of $p_k$. 
The cardinality  of the sample set thus changes at each fine iteration, and the increase is much slower at the coarse levels.

In order to define the objective functions for each level $\ell$, we introduce the following notation. Given a function $g$, we define $\quadre{g}^{\Sub}$ its subsampled version, i.e., if $g$ is the average of functions $g^{(i)}$ over a sample set including  $\Sub$, $\quadre{g}^{\Sub}$ is the average of the $g^{(i)}$ restricted to just the samples in  $\Sub$. If $g$ is not defined on a sample set, $\quadre{g}$ is just $g$.

We use the functions $\graffe{\fsl}_{\ell =1}^{\ell_{\max}}$ to define the regularized models that are minimized at each level in a recursive way.
In particular, at level $1<\ell\leq \ell_{\max}$, given the objective function of that level $f_k^\ell$ and an iterate $x_k$, we define the objective function $m_k^{R,\ell-1}$ at $x_k$ for the lower level $\ell-1$  as 
\begin{equation}\label{mRsf}
  m_k^{R,\ell-1}(s)=\quadre{\fl_k}^{\Sub_k^{\ell-1}}(x_k+s)+(v_k^{\ell-1})^Ts+\frac{1}{2}\lambda_k^{\ell}\norm{\nabla_x\fl_k(x_k)}\|s\|^2
\end{equation}
with $\lambda_k^\ell>0$ and
\begin{equation}\label{corr_vec}
  v^{\ell-1}_k = \nabla_x \fl_k(x_k)-\nabla_x \quadre{\fl_k}^{\Sub^{\ell-1}}(x_k).
\end{equation}

At the finest level $\fl_k=f^{\Sub_k^{\ell_{\max}}}$, thus $\quadre{f^{\Sub_k^{\ell_{\max}}}}^{\Sub_k^{\ell_{\max}-1}}$ is simply $f^{\Sub_k^{\ell_{\max}-1}}$. However, when $\ell<\ell_{\max}$, $\fl_k$ represents the regularized model built at the immediately upper level $\ell+1$ and incorporates also the regularization and the vector $v_k^{\ell+1}$. Given that this quantities are not defined on a samples set, the subsampled version of $f_k^\ell$ will contain the term $f^{\Sub_k^{\ell}}$ that is subsampled on ${\Sub_k^{\ell-1}}$, while the correction and the regularization vectors remain unchanged\footnote{Notice that each time we go down a level we accumulate in the regularized model a regularization term and a correction vector.}.

\begin{algorithm}[h!]
\caption{\name{} for finite-sum minimization - \name{}($\ell$,$x_0$,$f^\ell$,$\epsilon^{\ell}$,$\lambda^\ell$,$\mathcal{S}^{\ell}$)}\label{alg2: STREG}
\begin{algorithmic}[1]
\small
\Statex \textbf{Input:} {$\ell$ index of the level, $x_0 \in\R^{n}$ starting point, 
$f^\ell:\mathbb{R}^{n}\rightarrow\mathbb{R}$ objective function for level $\ell<\ell_{\max}$,  $\epsilon^{\ell}>0$ tolerance for the stopping criterion, $\lambda^{\ell}>0$, $\Sub^\ell$ sample set for the current level.
}
\Statex Given $0<\eta_1\leq\eta_3<1$, $\eta_2>0$, $0<\gamma_2\leq\gamma_1< 1<\gamma_3$.
\State Set $k=0$,  $\lambda_0^\ell=\lambda^\ell$ and $\mathcal{S}^{\ell}_0=\mathcal{S}^\ell$.
\While{the stop criterion for level $\ell$ is not satisfied}\label{alg2: stop_crit}
\Statex \textbf{Sample set construction}
\If{$\ell=\ell_{\max}$}
\State Set $p_k:=\min\{N,\max\{100k+n+2,\lambda_k^2\}\}$ and build $\Sub_k^{\ell_{\max}}\subseteq \{1,\dots,N\}$ drawing $p_k$ indices randomly. Set   $f_k^{\ell_{\max}}=f^{\Sub_k^{\ell_{\max}}}$ as in \eqref{coarse_model_l}.
\Else
\State Set $f_k^\ell=f^\ell$.
\EndIf
\Statex \textbf{Model choice} \label{alg2: model choice} 
\If{$\ell >1$}
\State Choose to go to Step \ref{alg2: Taylor step} or to Step \ref{alg2:  random subsamp}.
 \Else
 \State Go to Step \ref{alg2: Taylor step}.
 \EndIf
\Statex \textbf{Regularized Taylor step} 
\State \label{alg2: Taylor step} Define $m_{k}^\ell(s) = \fl_k(x_k)+\nabla_x\fl_k(x_k)^Ts$. Set $s_k = -\frac{\nabla_x\fl_k(x_k)}{\lambda^\ell_k\norm{\nabla_x\fl_k(x_k)}}$. Go to Step \ref{alg2: Acceptance step}.
\Statex \textbf{Sub-sampled model }
\State \label{alg2: random subsamp} Build $\Sub_k^{\ell-1}\subset \Sub_k^{\ell}$  randomly. Set $f^{\mathcal{S}_k^{\ell-1}}$ as in \eqref{coarse_model_l}.
\State \label{alg2: sub-level mod definition} Build the lower level approximation $f^{\ell-1}_k$ of $f^\ell_k$ from $f^{\mathcal{S}_k^{\ell-1}}$ (cf. discussion in section \ref{sec:multilevel}, below \eqref{corr_vec}.).  
Compute the correction vector $\vlmone$ as in \eqref{corr_vec}.
Define the lower level model $\varphi_k^{\ell-1}(s)$ and its regularization $m_k^{R,\ell-1}(s)$ as
\begin{align*}
   \varphi_k^{\ell-1}(s)& = f_k^{\ell-1}(x_k+s)+(v_k^{\ell-1})^Ts;\\
   	m_k^{R,\ell-1}(s) & = \varphi_k^{\ell-1}(s) +  \frac{1}{2}\lambda^\ell_k\norm{\nabla_x f_k^{\ell}(x_k)} \norm{s}^2.
\end{align*}
\State \textbf{Recursive call}
\State\label{alg2: recursive step} 
Choose $\epsilon^{\ell-1}$ 
and call 
\name{}($\ell-1$,$0$,$m_k^{R,\ell-1}$,$\epsilon^{\ell-1}$,$\lambda_k^\ell$,$\mathcal{S}_k^{\ell-1}$) to find an approximate solution $s^*$ of the problem 
\begin{equation*}
  \min_{s\in\R^n}m_k^{R,\ell-1}(s),
\end{equation*}
such that condition \eqref{stop_ell} is satisfied.
\State \label{alg2: set rec search direction} Set $s_k=s^*$ and $m_{k}^\ell(s) = \varphi_k^{\ell-1}(s)$. 
\Statex \textbf{Step acceptance of trial point}
\State \label{alg2: Acceptance step} Compute $\rho_k^\ell := \frac{\fl_k(x_k)-\fl_k(x_k+s_k)}{m_{k}^\ell(0)-m_{k}^\ell(s_k)}$.\label{defrho}
\If{$\rho^{\ell}_k\geq \eta_1$ and $\norm{\nabla_x \fl_k(x_k)} \geq \eta_2/\lambda^\ell_k$}\label{alg2: sol param update begin}
\State $x_{k+1} = x_k + s_k$
\Else \State $x_{k+1} = x_k$
\EndIf
\Statex \textbf{Regularization parameter update}
\If{$\rho_k^\ell\geq \eta_1$ and $\norm{\nabla_x \fl_k(x_k)} \geq \eta_2/\lambda^\ell_k$}\label{alg2: reg param update}
\State $$\lambda^\ell_{k+1} =
			\bigg \{
			\begin{array}{ll}
			\max\{\lambda_{\min},\gamma_2\lambda^\ell_k\},  & \text{ if }\rho_k^\ell\geq \eta_3, \\
			\max\{\lambda_{\min},\gamma_1\lambda^\ell_k\},  & \text{ if } \rho_k^\ell< \eta_3\\
			\end{array}
			$$ 
\Else
\State $\lambda^\ell_{k+1}= \gamma_3\lambda^\ell_k$. \label{alg2: sol param update end}
\EndIf
\State $k = k + 1$
\EndWhile
\end{algorithmic}
\end{algorithm}

The stopping criterion depends on the level. At fine level it checks if the norm of the approximated gradient is below some tolerance  $\epsilon$:
\begin{equation}\label{stop_crit_elllmax}
    \norm{\nabla_x f^{\mathcal{S}_k^{\ell_{\max}}}(x_k)
} \leq \epsilon.
\end{equation}
If the full gradient is not evaluated (i.e., $|\mathcal{S}_k^{\ell_{\max}}|<N$), case expected for most of the fine iterations, the stopping criterion (\ref{stop_crit_elllmax}) might not be meaningful. 
Therefore we use a heuristic stopping test. When \eqref{stop_crit_elllmax} is satisfied for the first time, after a fine or a coarse step, a new set of $p_k$ randomly chosen samples is drawn and fine steps are taken until  (\ref{stop_crit_elllmax}) is satisfied again. 
In practice the stopping criterion is however, in most cases,  satisfied when the full sample size is reached. 
A safeguard is also added that imposes a maximum  number of fine iterations. 
If $\ell<\ell_{\max}$, we use the stopping condition \eqref{stop_ell}
where $m_k^{R,\ell}$ is defined in \eqref{mRsf}, with $\epsilon^{\ell-1} >0$, and impose a maximum number $\rm{maxit}_\ell$ of iterations.

We report in Algorithm \ref{alg2: STREG} the complete \name{} algorithm
for problem  \eqref{pb_opti_fs}.
The call to the algorithm at fine level is  \name{}($\ell_{\max}$,$x_0$,-,$\epsilon^{\ell_{\max}}$,$\lambda_0^{\ell_{\max}}$,-), with $x_0\in\mathbb{R}^n$. 
Note that the choice of the alternate scheme between the coarse and fine steps is left to the user.

\section{Numerical experiments}\label{sec:num}

In this section we illustrate the performance of \name{} for the solution  of finite sum minimization problems (\ref{pb_opti_fs}) arising in  binary classification.

\paragraph{Considered methods}
We use 3 levels  (adding more levels did not improve the results) and we rename our method \nametre{}. We compare \nametre{} against the 1-level \nameuno{}. We remark that \nameuno{} can be interpreted  as a STORM-like approach
where the trust region constraint is replaced by a quadratic regularization, resulting in fact in an adaptive sampling strategy.

We also consider two first-order stochastic algorithms widely used in the solution of (\ref{pb_opti_fs}):
SVRG \cite{johnson2013accelerating} and Adagrad~\cite{duchi2011adaptive,mcmahan2010adaptive}.
The choice of SVRG is due to the  variance reduction interpretation of multilevel methods mentioned in the introduction, while Adagrad was chosen among the Ada-like methods for its good complexity properties \cite{gratton2024complexity}. As our approach, Adagrad uses an adaptive strategy for the stepsize, while SVRG uses a constant steplength (or learning rate) $\alpha$. Both algorithms use a mini-batch of size $b$ for evaluating the approximated gradient but SVRG also need to re-evaluate the full gradient every $m$  iterations \cite{johnson2013accelerating}.

\paragraph{Performance measures}
In order to compare the efficiency of the various methods, we consider the number of weighted gradient and function evaluations performed during the execution: a full-gradient evaluation is counted as 1, while a function evaluation is counted as $\frac{1}{n}$, taking into account that the size of the gradients is $n$.  Therefore, the  sub-sampled gradients are weighted as $\frac{\lvert \Sub_k^{\ell}\rvert}{N}$, the size of the sub-sample set. 
From now on,  we will refer to this measure as computational effort or more simply weighted number of evaluations, which will be denoted by \textbf{\#f/g}.

Beside the efficiency of the methods, we also take into account the quality of the solutions found.
In particular, we consider the classification accuracy (in percentage) on the testing set that will be denoted by \textbf{\%tA}.

\paragraph{Implementation issues } Both the versions of 
\name{} with one and three levels have been implemented following  Algorithm \ref{alg2: STREG} with parameters chosen as follows:
$$\eta_1= 0.5,\ \eta_2 = 10^{-3}, \ \eta_3=0.75, \ \gamma_1= 0.5,\ \gamma_2 = 0.3,\ \gamma_3=2,\ \lambda_{\min}= 10^{-4}.$$
  Moreover, for the 3 level version, we set $\lvert \Sub_k^1\rvert= 0.001\, p_k$ and $\lvert \Sub_k^2\rvert= 0.01\, p_k$.
The recursion scheme encompasses a fine step after each recursive call, as depicted in Figure \ref{fig:schema3} where the horizontal arrows represent the fine steps. 
\begin{figure}[h]
\centering
\includegraphics[width=0.8\textwidth]{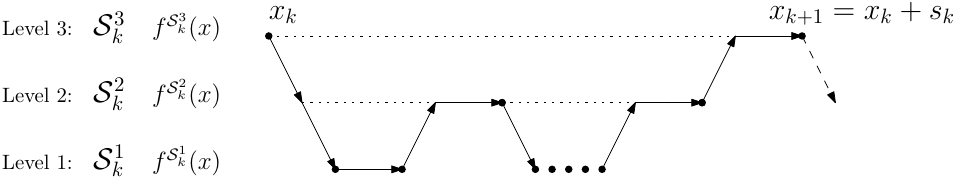}
\caption{Iteration scheme used in our implementation of \name\ for problem \eqref{pb_opti_fs}.} \label{fig:schema3}
\end{figure}

Algorithm \ref{alg2: STREG} terminates when condition
\eqref{stop_crit_elllmax} holds with $\epsilon = 10^{-3}$ and we set $\epsilon^{\ell-1} = 10^{-3}$ in \eqref{stop_ell} and  a maximum number of coarse iterations $\text{maxit}_{\ell}=5$.
 Finally, we set  $\lambda_0^\ell= 10^{-4}$ for $\ell>1$ and $\lambda_0^1 = 10^{-3}$.

For the implementation of SVRG and Adagrad, we followed \cite{reddi2016stochastic,ward2020adagrad}  and, also taking into account our own extensive tuning work, we set the mini-batch size $b=20$ for both methods and chose the learning rate $\alpha = 0.01$  and $m = N/b$ for SVRG. Regarding the stopping criterion, SVRG terminates when the norm of the gradient computed on the full sample set is below  $\epsilon = 10^{-3}$ or $10^4$ iterations are performed. Imposing a stopping criterion on Adagrad is less obvious and we only set a maximum number of weighted number of gradients and functions.

All algorithms  have been implemented in MATLAB R2024a using HPE ProLiant DL560 Gen10 with 4 Intel(R) Xeon(R) Gold 6140 CPU @ 2.30GHz with 512 Gb RAM\footnote{We kindly acknowledge the Department of Mathematics of the University of Bologna for making the department's HPC resources available for this work.}.

\paragraph{Results on binary classification problems} \label{sec:expfs}

We consider a binary classification problem where the objective function 
 is obtained by composing a least-squares loss with the sigmoid function, that is 
\begin{equation}\label{prob: nonlin least squares}
    \min_{x \in  \mathbb{R}^n}  F(x) =  \frac{1}{2N}\sum_{i=1}^N \tonde{y_i - \frac{1}{1 + \text{exp}(-y_ix^Tz_i)}}^2, \tag{Pb-FS}
\end{equation}
with $(z_i,y_i)\in \R^n \times \{0,1\}$, for every $i=1,\dots,N$. The tests are performed with four different datasets for binary classification: \mnist{} \cite{lecun1998gradient}, \mush{} \cite{mushroom_73}, \aninea{} and \ijcnnone{} \cite{CC01a}. 
The data sets are divided into a training set and a testing set as specified in Table \ref{tab: datasets}.
\begin{table}[h!]
    \centering
    \begin{tabular}{llll}
    Data set        & nr. of features ($n$)   & Training set size ($N$)& Testing set size ($N_t$) \\ \hline
    \mnist{}  & 784                     & 60000                  & 10000 \\
    \mush{}   & 112                     & 6503                   & 1621\\
    \aninea{}    & 123                     & 22793                  & 9768\\
    \ijcnnone{} & 22                      & 49990                  & 91701\\
    \end{tabular}
    \caption{Data sets with number of features $n$ and number of instances of the training set $N$ and the testing set $N_t$.}
    \label{tab: datasets}
\end{table}

For each dataset we perform five tests with random initial guesses and then we show the averaged values (see Table \ref{tab:tutti} and Figures  \ref{fig: obj. vs. ev.}-\ref{fig: acc vs. ev.}). 
In Table \ref{tab:tutti} we consider the computational effort  \textbf{\#f/g} of the solvers to  satisfy the convergence stopping criteria\footnote{Adagrad is not included, since the stopping criterion is based on a maximum number of function evaluations only.}.

\begin{table}[h!]
\centering
\footnotesize
\begin{tabular}{|lccc|}
\hline
\multicolumn{4}{|c|}{\mush{}}\\  
\hline
 & \multicolumn{1}{|c|}{SVRG }                   & \multicolumn{1}{c|}{\nameuno} & \nametre \\\hline

\multicolumn{1}{|l|}{\textbf{Avg. \%tA}}               & \multicolumn{1}{c|}{98.69}                                                   & \multicolumn{1}{c|}{97.68}                                                    & 97.74                                                    \\
\multicolumn{1}{|l|}{\textbf{StD \%tA}}                & \multicolumn{1}{c|}{0.25}                                                         & \multicolumn{1}{c|}{0.50}                                                     & 0.48                                                     \\
\multicolumn{1}{|l|}{\textbf{Avg. \#f/g}}              & \multicolumn{1}{c|}{341.89}                                                         & \multicolumn{1}{c|}{216.78}                                                   & 35.48                                                    \\
\multicolumn{1}{|l|}{\textbf{StD \#f/g}}               & \multicolumn{1}{c|}{722.65}                                                  & \multicolumn{1}{c|}{30.69}                                                    & 9.95                                                     \\ \hline
\multicolumn{4}{|c|}{\aninea{}}                                                                                                                                                                                                                                                                                     \\ \hline
 & \multicolumn{1}{|c|}{SVRG }                   & \multicolumn{1}{c|}{\nameuno} & \nametre \\\hline

\multicolumn{1}{|l|}{\textbf{Avg. \%tA}}               & \multicolumn{1}{c|}{85.00}                                                       & \multicolumn{1}{c|}{84.65}                                                    & 84.83                                                    \\
\multicolumn{1}{|l|}{\textbf{StD \%tA}}                & \multicolumn{1}{c|}{0.05}                                                          & \multicolumn{1}{c|}{0.04}                                                     & 0.07                                                     \\
\multicolumn{1}{|l|}{\textbf{Avg. \#f/g}}              & \multicolumn{1}{c|}{840.77}                                                & \multicolumn{1}{c|}{207.68}                                                   & 90.87                                                    \\
\multicolumn{1}{|l|}{\textbf{StD \#f/g}}               & \multicolumn{1}{c|}{300.51}                                                    & \multicolumn{1}{c|}{53.97}                                                    & 30.31                                                    \\ \hline
\multicolumn{4}{|c|}{\ijcnnone{}}                                                                                                                                                                                                                                                                                  \\ \hline
& \multicolumn{1}{|c|}{SVRG }                   & \multicolumn{1}{c|}{\nameuno} & \nametre \\\hline

\multicolumn{1}{|l|}{\textbf{Avg. \%tA}}               & \multicolumn{1}{c|}{91.68}                                                       & \multicolumn{1}{c|}{91.10}                                                    & 91.13                                                    \\
\multicolumn{1}{|l|}{\textbf{StD \%tA}}                & \multicolumn{1}{c|}{0.00}                                                           & \multicolumn{1}{c|}{0.25}                                                     & 0.09                                                     \\
\multicolumn{1}{|l|}{\textbf{Avg. \#f/g}}              & \multicolumn{1}{c|}{553.87}                                                       & \multicolumn{1}{c|}{6.20}                                                     & 7.09                                                     \\
\multicolumn{1}{|l|}{\textbf{StD \#f/g}}               & \multicolumn{1}{c|}{1.64}                                                              & \multicolumn{1}{c|}{0.97}                                                     & 0.64                                                     \\ \hline
\multicolumn{4}{|c|}{\mnist}                                                                                                                                                                                                                                                                                   \\ \hline
 & \multicolumn{1}{|c|}{SVRG }                   & \multicolumn{1}{c|}{\nameuno} & \nametre \\\hline

\multicolumn{1}{|l|}{\textbf{Avg. \%tA}}               & \multicolumn{1}{c|}{90.43}                                               & \multicolumn{1}{c|}{89.80}                                                    & 89.82                                                    \\
\multicolumn{1}{|l|}{\textbf{StD \%tA}}                & \multicolumn{1}{c|}{0.13}                                                     & \multicolumn{1}{c|}{0.04}                                                     & 0.05                                                     \\
\multicolumn{1}{|l|}{\textbf{Avg. \#f/g}}              & \multicolumn{1}{c|}{17843.40}                                                     & \multicolumn{1}{c|}{2923.02}                                                  & 414.15                                                   \\
\multicolumn{1}{|l|}{\textbf{StD \#f/g}}               & \multicolumn{1}{c|}{6336.56}                                                  & \multicolumn{1}{c|}{497.15}                                                   & 16.37   \\ \hline                                                
\end{tabular}
\caption{\eqref{prob: nonlin least squares} Comparison between \nameuno{}, \nametre{} and SVRG. Average of maximum classification accuracy reached and number of evaluations, with corresponding standard deviation. }
\label{tab:tutti}
\end{table}

We first observe that \nametre{} is in general much more efficient than the one-level version, while reaching the same accuracy level.
\begin{figure}[h!]
    \centering
    \includegraphics[width=\linewidth]{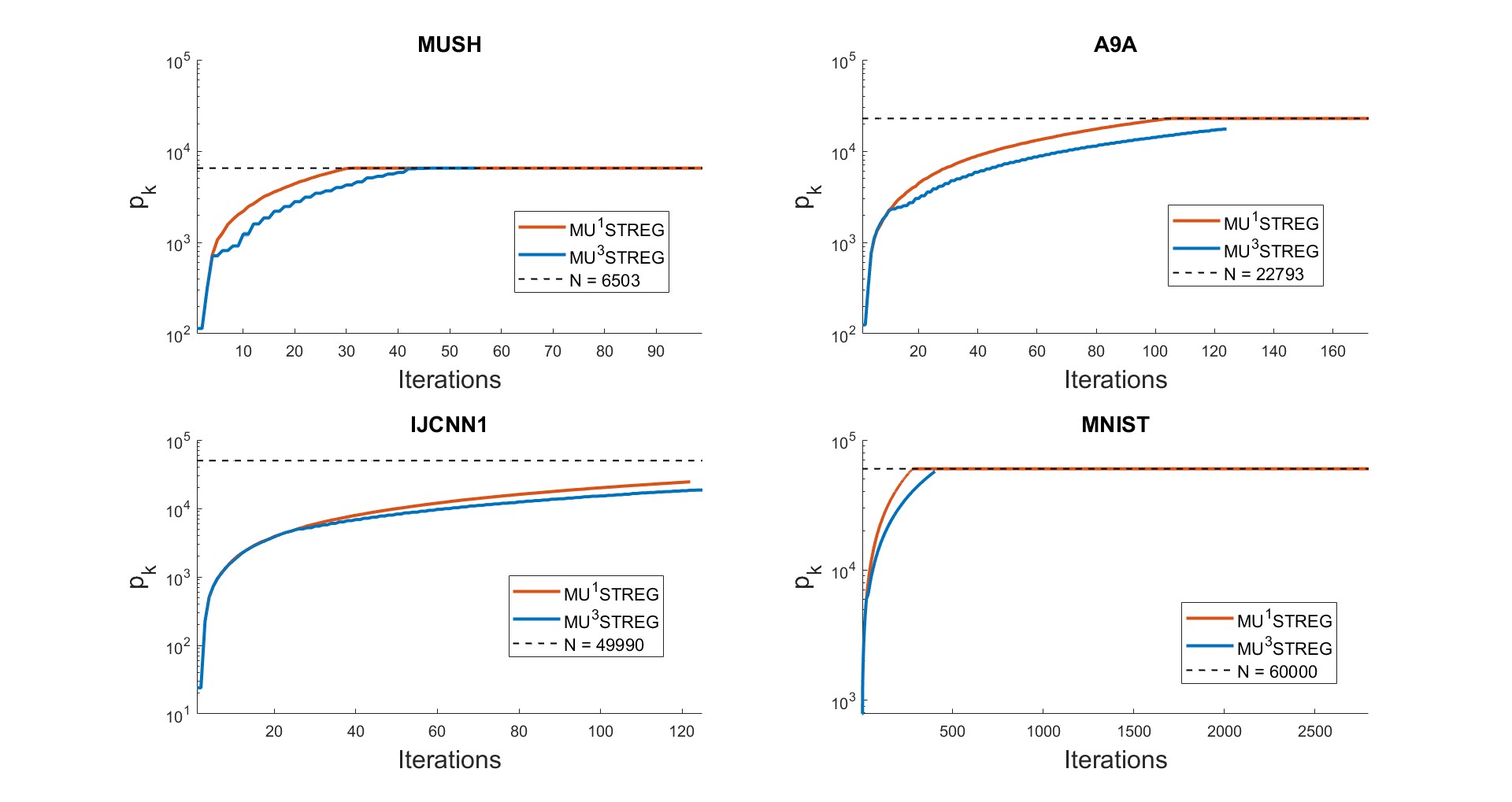}
    \caption{\eqref{prob: nonlin least squares} Cardinality of the sample set at the finest level for \nameuno{} and \nametre{} and the full size $N$ along the iterations.}
    \label{fig: pk vs. ev.}
\end{figure}
Also, Figure \ref{fig: pk vs. ev.} shows the evolution of the cardinality $p_k$ of the sample set at the finest level along the finest iterations, for one illustrative run of the five performed. 
We can observe that \nametre{} often requires a smaller sample set, especially in the first phase of the iteration history.

Moreover, both \name{} solvers requires definitely less  weighted number of evaluations than SVRG with similar solution quality. We also remark that the reported results for SVRG are the best obtained
after a consistent tuning work for the learning rate $\alpha$, while the stepsize selection is automatic in our approach.

We now compare our method with Adagrad, imposing a maximum budget of 100 weighted function and gradient evaluations \textbf{\#f/g}
\footnote{We do not consider SVRG in this analysis, as the required budget to get an accurate solution for SVRG is much larger than 100 and therefore
the plots of its performance curves are unreadable.}.
 Figures  \ref{fig: obj. vs. ev.} and \ref{fig: acc vs. ev.} show the value of the objective function and of the classification accuracy along \textbf{\#f/g}, respectively. Here, 
 the solid curves represent the mean values and the shaded area takes into account the standard deviation.
  Figure  \ref{fig: obj. vs. ev.} shows that the objective decrease is faster for \name{} solvers than for Adagrad, while 
  Figure \ref{fig: acc vs. ev.} highlights an oscillatory behaviour of Adagrad that reaches a lower accuracy than \name{} (especially for the \aninea and \mnist datasets) 
  and comparable in case of \mush.

\begin{figure}[h!]
    \centering
    \includegraphics[width=\linewidth]{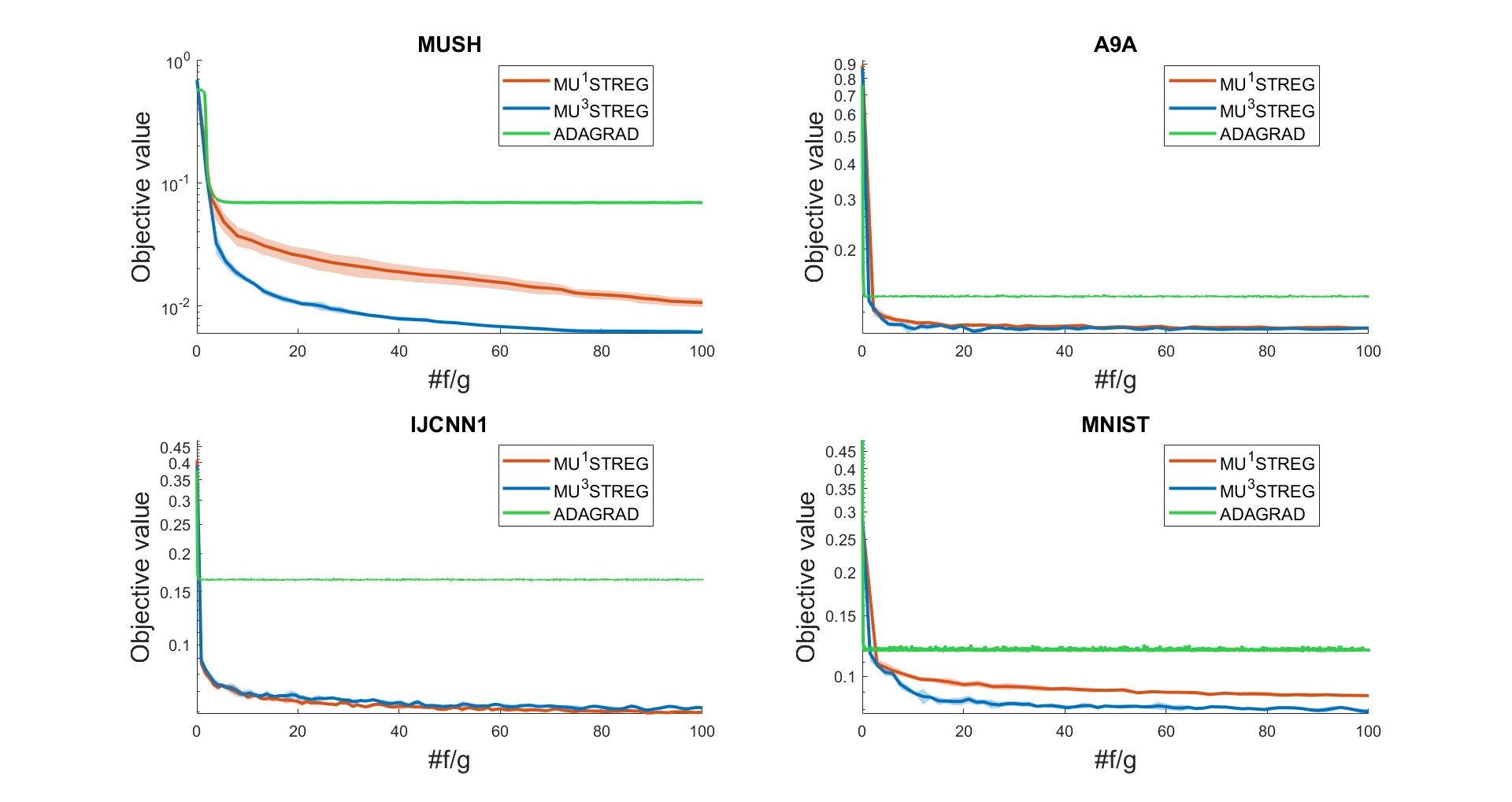}
    \caption{\eqref{prob: nonlin least squares} Objective function value along the number of  weighted evaluations of gradients and functions.}
    \label{fig: obj. vs. ev.}
\end{figure}
\begin{figure}[h!]
    \centering
    \includegraphics[width=\linewidth]{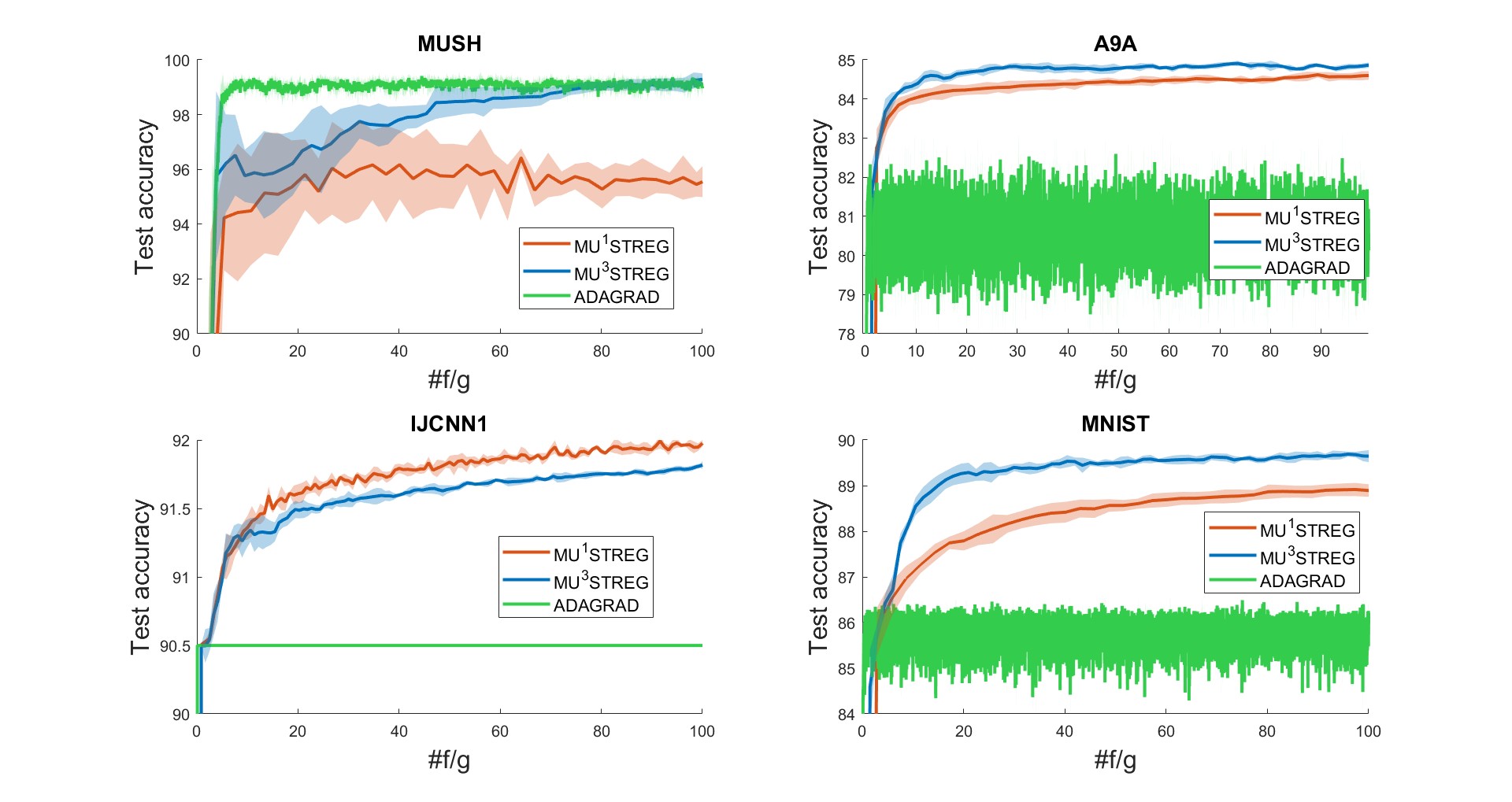}
    \caption{\eqref{prob: nonlin least squares} Classification accuracy on the testing set along the number of  weighted evaluations of gradients and functions.}
    \label{fig: acc vs. ev.}
\end{figure}


\section{Conclusions}\label{sec_concl}

We have proposed a new framework for the multilevel solution of stochastic problems, assuming that the stochastic objective function admits a hierarchical representation. Our framework encompasses both hierarchies in the variable space and in the function space, meaning that the function can be represented at different levels of accuracy.

We propose \name{}, a new multilevel stochastic gradient method based on adaptive regularization that generalizes the AR1 method \cite{book_compl} and we propose a stochastic convergence analysis for it. This convergence theory is the first stochastic convergence study for multilevel methods. 

We detailed our method for the solution of finite sum minimization problems and we made experiments on binary classification problems.
We show that the proposed multi-level method outperforms the adaptive sampling one-level counterpart and is competitive with the tested subsampling methods (Adagrad and SVRG).

In particular, the finite sum minimization setting allows us to show the main advantage  of a stochastic multilevel framework with respect to the classical deterministic one: it does not need the evaluation of the function/gradient over the full samples set along all the iterations.

Our framework covers different contexts (cf. the examples in section \ref{sec:2level}). Testing its effectiveness in other practical contexts (Montecarlo simulations, hierarchies in the variables space, notably, e.g., when random projection operators are used)  is an open research direction.

Moreover, our analysis assumes that the stochastic objective function $f$ admits a hierarchy of computable approximations, i.e., that the functions $\phi_k^\ell$, once drawn randomly, are deterministic. The development of a fully stochastic analysis, where the functions $\phi_k^\ell$ are assumed to be stochastic like $f$, is another meaningful perspective.


\section*{Acknowledgments}
{\footnotesize
The authors wish to thank the anonymous referee for his/her careful reading and suggestions, which led to significant improvement of the manuscript.  

The work of the first and second  author was partially supported by INdAM-GNCS under the INdAM-GNCS project CUP\_E53C22001930001. 
Part of the work of the F.M. was started during the author's Ph.D. thesis at the Universit\`a di Bologna supported by the program ``Programma Operativo Nazionale Ricerca e Innovazione 2014-2020 (CCI2014IT16M2OP005)" - Azione IV.5 "Dottorati e contratti di ricerca su tematiche green" XXXVII ciclo, code DOT1303154-4, CUP J35F21003200006.
The research of M.P.  was  partially granted by the Italian Ministry of University and Research (MUR) through the PRIN 2022 ``MOLE: Manifold constrained Optimization and LEarning'',  code: 2022ZK5ME7 MUR D.D. financing decree n. 20428 of Nov. 6th, 2024 (CUP B53C24006410006), and by PNRR - Missione 4 Istruzione e Ricerca - Componente C2 Investimento 1.1, Fondo per il Programma Nazionale di Ricerca e Progetti di Rilevante Interesse Nazionale (PRIN) funded by the European Commission under the NextGeneration EU programme, project ``Advanced optimization METhods for automated central veIn Sign detection in multiple sclerosis from magneTic resonAnce imaging (AMETISTA)'',  code: P2022J9SNP,
MUR D.D. financing decree n. 1379 of 1st Sept. 2023 (CUP E53D23017980001).
The work of E.R. was partially funded by the Fondation Simone et Cino Del Duca - Institut de France and  by MEPHISTO
(ANR-24-CE23-7039-01) project of the French National Agency for Research (ANR).
}

\bibliographystyle{plain}
\bibliography{biblio_Lyon}

\end{document}